\documentclass{article}
\usepackage{latexsym}
\usepackage{amssymb}
\usepackage{amsfonts}
\usepackage{mathrsfs}
\usepackage{ulem}
\usepackage{chemarrow}
\usepackage{amsthm,array}
\usepackage{pgf} \usepackage{pgflibrarysnakes}
\usepackage{cite}

\usepackage{tikz}

 \usepackage[colorlinks=true]{hyperref}

\usepackage{comment}
\usepackage{array}
\usepackage{color,graphicx}
\usepackage{pstricks}
\usepackage{amsmath}
\usepackage{enumerate}
\usepackage{multirow}
\usepackage{geometry}
\newcommand*\rot{\rotatebox{90}}

\newcommand*\circled[1]{\tikz[baseline=(char.base)]{
            \node[shape=circle,draw,inner sep=2pt] (char) {#1};}}
            
\newcommand*\squared[1]{\tikz[baseline=(char.base)]{
            \node[shape=rectangle,draw,inner sep=2pt] (char) {#1};}}

\def\vs{\vskip.0cm}
\def\br{\mathbb R}
\def\bz{\mathbb Z}
\def\bc{\mathbb C}
\def\wt{\widetilde}

\def\id{\text{\rm Id\,}}
\def\ve{\varepsilon}
\def\cV{\mathcal V}

\def\noi{\noindent}

\def\ve{\varepsilon}

\def\maGdeg{\mathcal G\text{\rm -deg}}
\newtheorem{theorem}{Theorem}[section]
\newtheorem{proposition}[theorem]{Proposition}

{\bfseries}{\rmfamily}
\newtheorem{definition}[theorem]{Definition}
\newtheorem{remark}[theorem]{Remark}

\newtheorem{remark-definition}[theorem]{Remark and Definition}

\newrgbcolor{lightblueA}{.50 1 1}
\newrgbcolor{violet}{.6 .1 .8}
\newrgbcolor{lightyellow}{1 1 .8}
\newrgbcolor{lightblue}{.80 1 1}
\newrgbcolor{mygreen}{0 .66 .05}
\definecolor{mygreen}{rgb}{0,.66,.05}
\definecolor{lightyellow}{rgb}{1,1,.80}
\newrgbcolor{orange}{1 .6 0}
\newrgbcolor{GreenYellow}{.85 1 .31}
\newrgbcolor{Yellow}{1  1  0}
\newrgbcolor{Goldenrod}{1  .90  .16}
\newrgbcolor{Dandelion}{1  .71  .16}
\newrgbcolor{Apricot}{1  .68  .48}
\newrgbcolor{Peach}{1  .50  .30}
\newrgbcolor{Melon}{1  .54  .50}
\newrgbcolor{YellowOrange}{1  .58  0}
\newrgbcolor{Orange}{1  .39  .13}
\newrgbcolor{BurntOrange}{1  .49  0}
\newrgbcolor{Bittersweet}{1.  .4300  .24}
\newrgbcolor{RedOrange}{1  .23  .13}
\newrgbcolor{Mahogany}{1.  .4475  .4345}
\newrgbcolor{Maroon}{1.  .4084  .5376}
\newrgbcolor{BrickRed}{1.  .3592  .3232}
\newrgbcolor{Red}{1  0  0}
\newrgbcolor{OrangeRed}{1  0  .50}
\newrgbcolor{RubineRed}{1  0  .87}
\newrgbcolor{WildStrawberry}{1  .04  .61}
\newrgbcolor{Salmon}{1  .47  .62}
\newrgbcolor{CarnationPink}{1  .37  1}
\newrgbcolor{Magenta}{1  0  1}
\newrgbcolor{VioletRed}{1  .19  1}
\newrgbcolor{Rhodamine}{1  .18  1}
\newrgbcolor{Mulberry}{.6668  .1180  1.}
\newrgbcolor{RedViolet}{.9538  .4060  1.}
\newrgbcolor{Fuchsia}{.5676  .1628  1.}
\newrgbcolor{Lavender}{1  .52  1}
\newrgbcolor{Thistle}{.88  .41  1}
\newrgbcolor{Orchid}{.68  .36  1}
\newrgbcolor{DarkOrchid}{.60  .20  .80}
\newrgbcolor{Purple}{.55  .14  1}
\newrgbcolor{Plum}{.50  0  1}
\newrgbcolor{Violet}{.98 .15 .95}
\newrgbcolor{RoyalPurple}{.25  .10  1}
\newrgbcolor{BlueViolet}{.84  .38  .98}
\newrgbcolor{Periwinkle}{.43  .45  1}
\newrgbcolor{CadetBlue}{.38  .43  .77}
\newrgbcolor{CornflowerBlue}{.35  .87  1}
\newrgbcolor{MidnightBlue}{.4414  .9259  1.}
\newrgbcolor{NavyBlue}{.06  .46  1}
\newrgbcolor{RoyalBlue}{0  .50  1}
\newrgbcolor{Blue}{0  0  1}
\newrgbcolor{Cerulean}{.06  .89  1}
\newrgbcolor{Cyan}{0  1  1}
\newrgbcolor{ProcessBlue}{.04  1  1}
\newrgbcolor{SkyBlue}{.38  1  .88}
\newrgbcolor{Turquoise}{.15  1  .80}
\newrgbcolor{TealBlue}{.1572  1.  .6668}
\newrgbcolor{Aquamarine}{.18  1  .70}
\newrgbcolor{BlueGreen}{.15  1  .67}
\newrgbcolor{Emerald}{0  1  .50}
\newrgbcolor{JungleGreen}{.01  1  .48}
\newrgbcolor{SeaGreen}{.31  1  .50}
\newrgbcolor{Green}{0  1  0}
\newrgbcolor{ForestGreen}{.1992  1.  .2256}
\newrgbcolor{PineGreen}{.3100  1.  .5575}
\newrgbcolor{LimeGreen}{.50  1  0}
\newrgbcolor{YellowGreen}{.56  1  .26}
\newrgbcolor{SpringGreen}{.74  1  .24}
\newrgbcolor{OliveGreen}{.6160  1.  .4300}
\newrgbcolor{RawSienna}{.53  .28  .16}
\newrgbcolor{Sepia}{1.  .7510  .70}
\newrgbcolor{Brown}{.41  .25  .18}
\newrgbcolor{TAN}{.86  .58  .44}
\newrgbcolor{Gray}{1.  1.  1.}
\newrgbcolor{Black}{1  1  1}
\newrgbcolor{White}{1  1  1}

\begin{document}
\title{Hopf Bifurcation of Relative Periodic Solutions:\\ Case Study of a Ring of Passively Mode-Locked Lasers}
\author{\bf Zalman Balanov,  Pavel Kravetc, Wieslaw Krawcewicz, \bf  Dmitrii Rachinskii\\
{\small Department of Mathematical Sciences, The University of Texas at Dallas,  Richardson, 75080 USA}}

\date{}

\maketitle

\begin{abstract}
In this paper, we consider an equivariant Hopf bifurcation of relative periodic solutions from relative equilibria in systems of functional differential equations respecting $\Gamma \times S^1$-spatial symmetries.
The existence of branches of relative periodic solutions together with their symmetric classification is established
using  the equivariant twisted $\Gamma\times S^1$-degree with 
one free parameter. As a case study, we consider a delay differential model of coupled identical passively mode-locked semiconductor lasers with the dihedral symmetry group $\Gamma=D_8$.
\end{abstract}


\medskip


\section{Introduction}

A natural counterpart of an equilibrium state in dynamical systems with continuous symmetries
is a {\it relative equilibrium}, {\it i.e.} an equilibrium modulo the group action. Similarly, a counterpart of a periodic solution is a {\it relative periodic solution}. 
In particular, in $S^1$-symmetric systems, the {\it $S^1$-equivariant Hopf bifurcation} is responsible for branching of relative periodic solutions from a relative equilibrium. This scenario is analogous to the classical Hopf bifurcation of periodic solutions from an equilibrium state in generic systems (without symmetry).

Relative equilibrium states and relative periodic motions are well-known for conservative systems related to rigid bodies \cite{MarsdenRatiu}, deformable bodies \cite{Chandrasekhar}, molecular vibrations \cite{KozinRobertsTennyson},  celestial mechanics \cite{Brouke,MeyerHall} and vortex theory\cite{PekarskyMarsden} 
(see also \cite{WulffRoberts,Golubitsky-Stewart} and references therein).
In addition to $S^1$-symmetry, many of such systems respect a finite group $\Gamma$ of spatial symmetries such as, for example, a symmetry of coupling of atoms in molecules \cite{MontaldiRoberts}. 
This naturally leads to the problem of classification of relative equilibria/periodic solutions according to their symmetric properties (spatio-temporal symmetries) represented by a subgroup $H$  of the group $\Gamma\times S^1$ for relative equilibria and the group  $\Gamma\times S^1\times S^1$ for relative periodic solutions, respectively, where the second copy of $S^1$ is associated with time periodicity. Examples include dynamics of a deformable body in an ideal irrotational fluid \cite{WulffRoberts},  symmetric celestial motions, for instance, central configurations  \cite{Meyer,RybickiPerez}, etc. On the other hand, there is a long list of applications described by {\it non-conservative} systems of ODEs admitting relative equilibria/relative periodic solutions (see, for example, 
\cite{Golubitsky-Stewart}, where the Couette-Taylor experiment is discussed in detail). 

In his pioneering work \cite{Krupa} (cf.  \cite{Field1}), M. Krupa proposed a general method for analysis of the bifurcation of relative periodic solutions from a relative equilibrium for systems of ordinary differential equations (in general, non-Hamiltonian). This elegant method reduces the problem to the analysis of a generic (non-symmetric) Hopf bifurcation for an explicit differential equation on the normal slice to the relative equilibrium. An extension of Krupa's method to the case of more complicated spatial symmetries (including 
$\Gamma \times S^1$) has been developed in \cite{Vanderbauwhede} (see also \cite{ChossatLauterbach} for the moving frames method and \cite{Lamb} for the hierarchy of secondary bifurcations).
Essentially, the analysis of the Hopf bifurcation includes two main problems: (i) finding the bifurcation points and establishing the occurrence of the bifurcation (in equivariant setting, this problem additionally requires to describe symmetric properties of the bifurcating solutions), and (ii) analysis of  stability properties of the bifurcating solutions. Krupa's method allows one to solve both problems for $S^1$-equivariant (or $\Gamma\times S^1$-equivariant, if an additional group $\Gamma$ of spatial symmetries is involved) ordinary differential systems using the center manifold reduction and the analysis of the normal forms of the bifurcation for the system on the normal slice.

The applicability of Krupa's method to $S^1$-symmetric functional differential equations (FDEs) appears to be problematic.
In particular, it is unclear whether it is possible to reduce the problem in question to the analysis of a non-equivariant FDE
on some kind of a normal slice. However, if the problem is limited to establishing the occurrence
of the bifurcation (problem (i) above), then alternative methods are available.
In the present paper, we address this more modest problem. In this case, one can
reduce the analysis to studying a fixed point problem for a $\Gamma\times S^1 \times S^1$-equivariant operator equation in a functional space of periodic functions.
For the latter problem, 
one can adapt either equivariant Lyapunov-Schmidt reduction techniques or $S^1$-equivariant topological methods (which is the case in this paper).

More specifically, we 
obtain conditions for the occurrence of the Hopf bifurcation of relative periodic solutions (together with their complete symmetric classification) from a relative equilibrium in general $\Gamma\times S^1$-equivariant systems of FDEs using the method based on twisted equivariant degree with one free parameter. For a systematic exposition of this method, we refer to \cite{AED,BK-H,survey,BKR,Ize-Vignoli,Ize, Wu}.
As is well-known, this method is insensitive to violations of genericity assumptions \cite{Fied3, HBKR} (these assumptions include the 
simplicity of purely imaginary eigenvalues at the bifurcation point, transversality 
of the eigenvalue crossing, and non-resonance conditions). Our results are formulated for a general FDE system, which respects a group $\Gamma\times S^1$ of spatial symmetries with an arbitrary finite group $\Gamma$, 
and include the method of classification of symmetries of the relative periodic solutions based on the linearization of the problem.
In the second part of the paper, these results are applied for the analysis of delay differential equations (DDE) describing a system
of 8 mode-locked lasers coupled in $D_8$-symmetric fashion.
The rate equations of semiconductor lasers are $S^1$-symmetric
(see, for example, \cite{genlaser1, genlaser2, genlaser3})
and can naturally include delays, the classical example being the Lang-Kabayashi model where the light is re-injected into the laser cavity by an external mirror \cite{LK, Yanchuk-Sieber}.

Semiconductor mode locked lasers are compact reliable low-cost devices that emit short optical pulses at high repetition rates suitable for applications in telecommunications, machining, probing, all optical systems, etc. \cite{chapter}. We use a delay differential model of a passively mode-locked laser proposed 
	in \cite{VT}. This delay differential system has been extensively 
	applied to analyze instabilities \cite{instab,instability, instability1,pulseinter} and hysteresis \cite{hysteresis, hysteresis1} in mode-locked lasers, optically injected lasers \cite{injection, injection1}, hybrid mode locking \cite{hybrid}, noise reduction \cite{noise}, 
	resonance to delayed feedback \cite{amann}, and Fourier domain mode locking \cite{fourier}. An increasing interest in small systems of symmetrically coupled lasers and large synchronized laser arrays  
	motivates the analysis of the corresponding variants of the model \cite{ls1,ls2,ls3,ls5}. 

In the mode-locking regime, the laser emits a periodic sequence of light pulses with the period close to the cold cavity round trip time, which equals to the delay in the model. A typical bifurcation scenario associated with formation of this regime is the Hopf bifurcation of a relative equilibrium (continuous wave solution) from the ``laser off'' equilibrium followed by the $S^1$-equivariant Hopf bifurcation of a relative periodic solution from the relative equilibrium with the increase of the bifurcation parameter (pump current). As the bifurcation parameter increases further,
the relative periodic solution continuously transforms to acquire a pulsating shape. This transformation is simultaneous with a sequence of secondary Hopf bifurcations from the equilibrium and relative equilibrium solutins. 

The paper is organized as follows. In the next short section, some equivariant jargon and notation are introduced.
In Section 3, we first classify symmetries of branches of relative equilibria, which bifurcate from
a $\Gamma\times S^1$-fixed equilibrium of an equivariant FDE. Then, the main theorem on classification of symmetries of relative periodic solutions bifurcating from the branches of relative equilibria (Theorem \ref{t1})
is presented and proved. In Section 4, the results of Section 3 are applied to delay rate equations of the $D_8\times S^1$-symmetric laser system. 
We prove the occurrence of infinitely many branches of relative equilibria with various symmetries from the laser off state.
Then the analytic method is combined with numerical computations 
to analyze symmetric properties of relative periodic solutions that branch from the relative equilibrium states.
A short appendix (Section 5) lists a few twisted subgroups, which are used in Section 4 to describe symmetries of solutions. 

\section{Equivariant Jargon}
Consider a compact Lie group  $G$. In what follows we will always assume that all considered in this paper subgroups $H\subset G$ are closed and denote by $N(H)$ the normalizer of $H$ in $G$, by $W(H)=N(H)/H$ the Weyl group of $H$ in $G$, and by $(H)$
the conjugacy class of $H$ in $G$. We will denote by  $\Phi(G)$ the set of all conjugacy classes of subgroups  in $G$. Clearly, $\Phi(G)$  admits a partial order defined by:
\[
(H) \leq (K)\quad \Leftrightarrow\quad \exists_{g\in G}\;\; gHg^{-1} \subset K.
\]
We also put $\Phi_n(G) := \{(H) \in \Phi(G): \; \dim W(H) = n\}$.
\vs

Let $X$ be a $G$-space and $x\in X$. We denote by  $G_x := \{g \in G: \; gx = x\}$ the {\it isotropy} (or {\it stabilizer})  of $x$, by $G(x):= \{gx :\; g \in G\} \simeq G/G_x$ the {\it orbit} of $x$, and by $X/G$ the orbit space of $X$. The conjugacy class $(G_x)$ will be called the {\it orbit type} of $x$. We will also adopt the following notation:
\begin{align*}
\Phi(G;X) &:=\{(H) \in \Phi(G) : \; H = G_x \; \text{for some} \; x \in X\},\\
\Phi_n(G;X)&:= \Phi_n(G) \cap \Phi(G;X),\\
X_H &:= \{x \in X: \; G_x = H \},\\
X^H &:= \{x \in X: \; G_x \supset H\},\\
X_{(H)} &:= \{x \in X: \; (G_x) = (H)\},\\
X^{(H)} &:= \{x \in X: \; (G_x) \geq (H)\}.
\end{align*}
One can easily verify that $W(H)$ acts on $X^H$ and this action is free on $X_H$.
\vs
For two $G$-spaces  $X$ and $Y$, a continuous map $f : X \to Y$ is said to be {\it equivariant} if $f(gx) = gf(x)$
for all $x \in X$ and $g \in G$. If the $G$-action on $Y$ is trivial, then $f$ is called {\it invariant}. Clearly, for any subgroup $H \subset G$ and equivariant map $f : X \to Y$, the map $f^H : X^H \to Y^H$, with $f^H:=f|_{X^H}$, is
well-defined and $W(H)$-equivariant.
\vs
Finally, given two orthogonal $G$-representations $W$ and $V$ and an open bounded subset
$\Omega \subset W$, a continuous map $f : W \to V$ is called {\it $\Omega$-admissible} if
$f(x)\not=0$ for all $x \in \partial \Omega$ (in such a case we will also say that $(f,\Omega  )$ is an {\it admissible pair}). Similarly, two $\Omega  $-admissible maps $f$ and $g$ are called $\Omega ${\it -admissible homotopic} if
there exists a continuous map $h:[0,1] \times W \to V$ (called $\Omega$-{\it admissible homotopy} between $f$ and $g$) such that: (i) $h(0,x) = f(x)$, $h(1,x) =g(x)$ for all $x\in W$, and (ii) the map $h_t(x) := h(t,x)$ is $\Omega$-admissible for all $t\in [0,1]$. We denote by $\mathcal M(W,V)$ the set of all admissible pairs
$(f,\Omega)$.
\vs
For further details of the equivariant jargon used in this paper,
we refer to \cite{tD,Kawa,Bre,AED}.

\section{$\Gamma\times S^1$-Symmetric Systems of FDEs}\label{sec:FDEs}
\subsection{Notation and statement of the problem}
 \noi Assume that $\Gamma$ is a finite group and
let $V:=\mathbb R^n$ be an orthogonal $\Gamma\times S^1$-representation such that the $S^1$-action on $V$ is given by the homomorphism  $T:S^1\to O(n)$. Assume that  $J$ is the infinitesimal operator of the subgroup $T(S^1)\subset O(n)$, i.e.
\[
J= \lim_{\tau\to 0} \frac1\tau \Big[T(e^{i\tau})-\text{Id}   \big].
\]
The action of $S^1$ on $V$ satisfies for all $e^{i\tau}\in S^1$
\[
\forall_{v\in V}\;\;\; e^{i\tau} v = e^{\tau J}(v),
\]
and we also have $Je^{\tau J}=e^{\tau J}J$.
\vs
We will denote $\mathcal G:=\Gamma\times S^1$ and use the notation
\begin{align*}
\bold \Gamma :=\Gamma\times \{1\} \quad \text{ and }\quad \bold S :=\{e\}\times S^1,
\end{align*}
where $e\in \Gamma$ is the neutral element.
\vs
Let
\begin{equation}\label{isot}
V=V_0\oplus V_1\oplus \dots \oplus V_m
\end{equation}
be the $\bold S$-isotypical decomposition of $V$, where $V_k$ is modeled on the $S^1$-irreducible representation $\mathcal V_k\simeq \bc$ with the $S^1$-action given by $e^{i\tau}z:=e^{ik\tau}\cdot z$, where `$\cdot$' stands for the complex multiplication.
Then, each of the components $V_k$, $k>0$, has a natural complex structure such that for $v\in V_k$
\[
e^{i\tau}v=e^{ik\tau}\cdot v, \quad Jv = i k \cdot v.
\]
Also, for $v\in V_0$, we have $Jv=0$.
\vs
Let $r>0$ and denote by $C_{-r}(V)$ the Banach space
\[C([-r,0];V):=\big\{ x: \text{ where } x: [-r,0]\to V \; \text{ is a continuous function}\big\},\]
equipped with the norm  $\|x\|_\infty:=\sup\{|x(\theta)|: \theta\in [-r,0]\}$. Clearly, $C_{-r}(V)$ is an isometric $\Gamma\times S^1$-representation, with the action given for by
\[
\forall_{\theta\in [-r,0]}\;\;((\gamma,e^{i\tau})x)(\theta)=e^{\tau J}(\gamma x(\theta)),\;\; x\in C_{-r}(V), \;\;  (\gamma,e^{i\tau})\in \Gamma\times S^1.
\]
In addition, we have the following $\bold S$-isotypical decomposition of  $C_{-r}(V):$
\[
C_{-r}(V)=\bigoplus_{k=0}^m C_{-r}(V_k),
\]
where each of the components $C_{-r}(V_k)$ with $k>0$ has a natural complex structure induced from $V_k$.
\vs
For a continuous function $x:\br\to V$ and $t\in \br$, let $x_t : [-r,0]\to V$ 
be a function defined by
\[ x_t(\theta):=x(t+\theta),\quad \theta \in [-r,0].
\]
We make the following assumption:
\begin{itemize}
\item[(A0)]
 $f:\mathbb R\times C_{-r}(V)\to V$ is  a $C^1$-differentiable $\mathcal G$-equivariant functional, i.e. for $(\gamma,e^{i\tau})\in \mathcal G$
\begin{equation}\label{eq:action}
f\Big(\alpha, (\gamma, e^{\tau i})x\Big)=e^{\tau J}\gamma f(\alpha, x)\quad \text{ for all }\; x\in C_{-r}(V).
\end{equation}
\end{itemize}

\subsection{Symmetric bifurcation of relative equilibria from an equilibrium}\label{subsec-relative-equib}
Consider the  parametrized system of FDEs
\begin{equation}\label{eq:fde1}
\dot x(t) =f(\alpha,x_t), \qquad x(t)\in V.
\end{equation}
In what follows, we will study bifurcations of continuous branches of periodic/quasi-periodic solutions to \eqref{eq:fde1} of special type and describe their symmetric properties. 

In this subsection, we are interested in 
periodic  solutions to \eqref{eq:fde1}  of the type
\begin{equation}\label{eq:part-sol}
x(t)=e^{w J t}x \quad \text{ for some } \;\; x\in V\;\text{ and }\; w\in\br.
\end{equation}

\paragraph{\bf Relative equilibria.} 
By substituting \eqref{eq:part-sol} into equation \eqref{eq:fde1}, we obtain
\begin{equation}\label{eq:step1}
w e^{wt J}Jx(t) = f(\alpha, e^{w(t+\cdot )J}x).
\end{equation}
Then, using the equivariance condition \eqref{eq:action}, we can  rewrite \eqref{eq:step1}  as
\begin{equation}\label{eq:step1'}
w Jx = f(\alpha, e^{w J\cdot}x),\quad x\in V.
\end{equation}

Take the orthogonal $\bold S$-invariant decomposition \eqref{isot} of the space $V$, where $V_0=V^{\bold S}$, and denote
\[
V_*:=V_0^\perp=V_1\oplus\dots\oplus V_m.
\]
For a fixed $\lambda =u+i w\in \mathbb{C}$, define the linear operator $\xi(\lambda):V\to C_{-r}(V)$ by
\begin{equation}\label{eq:xi1}
\Big(\xi(\lambda)x\Big)(\theta)=e^{u\theta} e^{wJ\theta} x_*+x_o, \quad \theta\in [-r,0],
\end{equation}
where $x=x_*+x_o$, $x_*\in V_*$ and $x_o\in V_0$, and consider the                                        
function $\wt f:\br\times \bc\times V\to V$ defined by
\begin{equation}\label{d}
\wt f(\alpha,\lambda,x):= f(\alpha, \xi(\lambda)x).
\end{equation}
With this notation, equation \eqref{eq:step1'} can be written as
\begin{equation}\label{eq:fde-reduced}
w Jx = \wt f(\alpha, i w, x), \quad x\in V.
\end{equation}
Furthermore, assumption (A0) implies $\mathcal G$-equivariance of $\wt f$:
\[
\wt f\Big(\alpha, \lambda, (\gamma, e^{\tau i})x\Big)=e^{\tau J}\gamma \wt f(\alpha,\lambda, x)\quad \text{ for all }\; x\in V, \ \lambda\in\mathbb{C}.
\]
Hence, solutions to \eqref{eq:fde-reduced} come in $\bold S$-orbits. It is clear that any $\bold S$-orbit, which is a solution to
\eqref{eq:fde-reduced}, is flow invariant for \eqref{eq:fde1} and therefore, satisfies the standard definition of relative equilibrium (see, for example,
\cite{GS}). In what follows, we refine this concept to the setting relevant to our discussion.
\vs

Assume that $x \in V^{\bold S}=V_0$. Then, one has
$\wt f(\alpha,iw, x)=\wt f(\alpha,0,  x)$ for all $w$. Hence, any  solution $x \in V^{\bold S}$ of \eqref{eq:fde-reduced} is an equilibrium for equation \eqref{eq:fde1} with $\bold S(x)  =  \{x\}$.
On the other hand, solutions of \eqref{eq:fde-reduced} that satisfy $x \not\in V^{\bold S}$ form one-dimensional orbits.
\vs

\begin{definition}\label{defrel} \rm
    Suppose that \eqref{eq:fde-reduced} holds for some $\alpha_o, w_o\in \mathbb R$ and $x_o\not \in V^{\bold S}$. Then, the
     orbit $\bold S (x_o)$ is a one-dimensional curve in $V$ called a {\it relative equilibrium} of equation \eqref{eq:fde1}.
    \begin{itemize}
    \item[(i)] For $w_o\ne 0$, this orbit is a trajectory of time-periodic solutions
       $x(\cdot)=e^{(w_o \cdot+\tau)J}x_o$, $e^{i\tau} \in S^1$,
      to equation \eqref{eq:fde1}
      called a {\it rotating wave}.

      \item[(ii)]  For $w_o=0$, the relative equilibrium consists of equilibrium points  $e^{\tau J}x_o$, $e^{i\tau} \in S^1$, of \eqref{eq:fde1} (the so-called {\it frozen wave}).
      \end{itemize}
\end{definition}

\paragraph{\bf Characteristic quasi-polynomials.}
Let $\alpha_o \in \br$
be given and  let $x_o \in V^{\mathcal G} $ be an equilibrium for \eqref{eq:fde1}.
We will also call the pair $(\alpha_o,x_o)$ an {\it equilibrium}, or a {\it stationary solution}, in this case.
\vs

Let us consider the bifurcation of relative equilibria from this equilibrium.
Let  
\begin{equation}\label{eq:linearization-matrix}
{D}_{\bold x} f(\alpha,{\bold x}): \br\times C_{-r}(V)\to  V
\end{equation} denote the derivative
of the functional $f$ with respect to ${\bold x}\in C_{-r}(V)$.
If $x_o\in V_0$, then the Jacobi matrix $D_x \wt f(\alpha, \lambda ,x_o): V\to V$ , which is given by
\begin{equation}\label{eq:matrix-linearizartion}
D_x \wt f(\alpha, \lambda ,x_o)=D_{\bold x}f(\alpha,\xi(\lambda)x_o)\xi(\lambda)
\end{equation}
is $\bold S$-equivariant. Therefore, the subspaces $V_0$ and $V_*$ are  $\bold S$-invariant for this matrix.
Consider the restrictions
$D_x \wt f(\alpha,\lambda,x_o)|_{V_0}$ and $D_x \wt f(\alpha,\lambda,x_o)|_{V_*}$ 
and define the {\it characteristic quasi-polynomials} for $x_o\in V_0$ and $\lambda \in \bc$:
\begin{align}\label{eq:quasipolynomials}
\begin{split}
\mathcal P_{0}(\alpha,\lambda, x_o)&:={\rm det}_{\bc}\, \left( D_x \wt f(\alpha,\lambda,x_o)|_{V_0}-\lambda\, \id \right),\\
\mathcal P_{*} (\alpha,\lambda, x_o)&:={\rm det}_{\bc}\, \left( D_x \wt f(\alpha,\lambda,x_o)|_{V_*}-\lambda\, \id \right).
\end{split}
\end{align}

\vs
We make the following assumption.

\begin{itemize}

\item[(A1)]
At the equilibrium point $(\alpha_o,x_o)$, where $x_o\in V^{\mathcal G}$, the characteristic quasi-polynomial ${\mathcal P}_{0}(\alpha_o,\cdot, x_o)$ has no zero roots,
i.e. ${\mathcal P}_{0}(\alpha_o,0, x_o)\ne 0$.
\end{itemize}
\vs
By assumption (A1) and implicit function theorem, there exists a $C^1$-function $x:(\alpha_o-\ve,\alpha_o+\ve)\to V_0$ (for some $\ve>0$) such that $x(\alpha_o)=x_o$ and $\{(\alpha,x(\alpha)): \alpha\in (\alpha_o-\ve,\alpha_o+\ve)\}$ 
is a curve of equilibrium points for \eqref{eq:fde1}. By equivariance of $\wt f$, and due to $x_o\in V^{\mathcal G}$, it follows that $x(\alpha)\in V^{\mathcal G}$,
consequently the $2$-manifold $M\subset \br^2\times V^{\mathcal G}$ given by
\begin{equation}\label{eq:trivial-HB}
M:=\{(\alpha,w,x(\alpha)): \alpha\in (\alpha_o-\ve,\alpha_o+\ve),\; w\in \br\}
\end{equation}
 is composed of  solutions to \eqref{eq:fde-reduced}, which can be called {\it trivial}.
 \vs

The roots of the the quasi-polynomials are called the {\it characteristic roots} for $(\alpha_o,x_o)$. Assume that:
\vs
\begin{itemize}
\item[(A2)]
The quasi-polynomial ${\mathcal P}_{*}(\alpha_o,\cdot, x_o)$ has a characteristic root $\lambda=i w_o$ for some $w_o\in\mathbb R$
at the equilibrium point $(\alpha_o,x_o)$, but for any other equilibrium $(\alpha,x)$ from a neighborhood of $(\alpha_o,x_o)$ in $\mathbb{R}\times V_0$,  the 
corresponding characteristic polynomial 
has no roots of the form $\lambda=i w$, $w\in \mathbb R$.

\end{itemize}

 In order to find {\it nontrivial} solutions to \eqref{eq:fde-reduced} bifurcating from $M$, consider the equation
 \begin{equation}\label{eq:bif-HB}
 \Phi(\alpha,w,x):= \wt f(\alpha, i w,x)-wJx=0
 \end{equation}
 as a $\mathcal G$-symmetric bifurcation problem with two free parameters $\alpha$ and $w$.

By applying the standard terminology (see \cite{AED}),  if  $D_x\wt f(\alpha,i w,x)-wJ:V\to V$ is not an isomorphism, we call the point $(\alpha,w,x)\in M$ an $M${\it -singular} point of $\Phi$.
By the implicit function theorem, a necessary condition for a point $(\alpha',w',x')\in M$ to be a bifurcation point for equation \eqref{eq:fde-reduced} is that it is an $M$-singular point.
 Assumption (A2) implies that the point $(\alpha_o,w_o,x_o)$ satisfies this necessary condition.

\begin{remark}\label{rem-isolated-center}
{\rm Due to equivariance of $f$, the linearization of the delayed system \eqref{eq:fde1}
at the equilibrium point $(\alpha,x(\alpha))$ has two flow invariant subspaces
$C_{-r}(V_0)$ and $C_{-r}(V_*)$. Assumption (A1) means that
$\lambda=0$ is not an eigenvalue for the restriction of the linearization to
$C_{-r}(V_0)$. Assumption (A2) means that the restriction of the linearization
to $C_{-r}(V_{\ast})$ has a pair of eigenvalues $\lambda=\pm i w_o$ for $\alpha=\alpha_o$
and has no eigenvalues of the form $i w$, $w\in\mathbb{R}$, for $\alpha\ne\alpha_o$ sufficiently close to $\alpha_o$.}
\end{remark}

\paragraph{\bf Sufficient condition for bifurcation of relative equilibria.}
In order to provide a sufficient condition for the bifurcation of relative equilibria from
the point  $(\alpha_o,w_o,x_o)$ and  an equivariant topological classification of the bifurcating branches, we apply the twisted $ \mathcal G$-equivariant degree with one free parameter (for more details, see \cite{AED}). To be more precise,
consider the $\mathcal G$-isotypical decomposition of $V$ (see \eqref{isot}):
\begin{equation}\label{isot-Gamma-S1}
V = V_0 \oplus V_{\ast} =  \bigoplus_{i} V_{i}^0 \oplus \bigoplus_{j,k} V_{j,k} \quad(i = 0,...,r;\; j=0,...,s; \;\; k = 1,...,m),
\end{equation}
where  $V_{j,k}$ is the isotypical component modeled on the irreducible $\mathcal G$-representation $\mathcal V_{j,k}$ and 
$V_{i}^0$ can be identified with the $\bold \Gamma$-re\-pre\-sen\-tation
modeled on an irreducible  $\bold \Gamma$-representation $\mathcal V_i$.

\begin{remark}\label{rem-maximal-twisted-isot-comp}
\rm
Let $(H_o)$ be a maximal twisted orbit type in $V$. Then, $(H_0)$ is also a maximal twisted orbit type for 
some $V_{j_o,k_o}$ in \eqref{isot-Gamma-S1}, $k_o > 0$. In fact, if $U$ is a direct sum of two $\mathcal G$-representations $U_1$ and $U_2$, then
$\mathcal G_{(x,y)} = \mathcal G_x \cap \mathcal G_y$ for any $(x,y) \in U$, $x \in U_1$, $y \in U_2$.  

\end{remark}

\vs

For any  $j=0,...,s$ and $k = 1,...,m$, put
\[
\mathcal P_{j,k} (\alpha,\lambda):={\rm det}_{\mathbb C}\, \left( D_x \wt f(\alpha,\lambda,x(\alpha))|_{V_{j,k}}-\lambda\, \id \right),\qquad \lambda\in \mathbb C.
\]
Notice that the characteristic equation at $(\alpha,x(\alpha))$ can be written as
\begin{equation}\label{eq:char-1}
\mathcal P_* (\alpha,\lambda) :=   \mathcal P_* (\alpha,\lambda,x(\alpha)) =   \prod_{k>0}\prod_{j=0}^s \mathcal P_{j,k} (\alpha,\lambda)=0.
\end{equation}
This implies that $\lambda$ is a characteristic root for $(\alpha,x(\alpha))$ if it is a root of $P_{j,k} (\alpha,\lambda)=0$ for some $k>0$ and $j\ge 0$.
\vs

To formulate our first bifurcation result, we need two additional concepts. 
Observe that using (A2), one can choose a small neighborhood 
$Q$  of the point $iw_o$ in the right half-plane ${\rm Re}\,\lambda>0$
of $\bc$ and a sufficiently small real $\delta=\delta(Q)>0$ such that,
as $\alpha$ varies over the interval $|\alpha-\alpha_o|\le\delta$, the roots $\lambda(\alpha)$ of $\mathcal P_{j,k}(\alpha,\cdot)$ can only leave $Q$ through the `exit' at the point $iw_o$ and only when $\alpha=\alpha_o$. 
\begin{definition}\label{isotyp-crossing-number} 
{\rm Define  the $V_{j,k}$-isotypical {\it crossing number} at $(\alpha_o,w_o)$ by the formula
\begin{equation}\label{crossingnumber}
\mathfrak t_{j,k}(\alpha_o,w_o):={\mathfrak t}^-_{j,k}(\alpha_o,w_o)-{\mathfrak t}^+_{j,k}(\alpha_o,w_o),
\end{equation}
where ${\mathfrak t}^-_{j,k}(\alpha_o,w_o)$ is the number of roots $\lambda(\alpha)$ of $\mathcal P_{j,k}(\alpha,\cdot)$ (counted according to their $\mathcal V_{j,k}$-isotypical multiplicity) in the set $Q$ for $\alpha<\alpha_o$, and ${\mathfrak t}^+_{j,k}(\alpha_o,w_o)$ is the number of roots of
$\mathcal P_{j,k}(\alpha,\cdot)$ in $Q$ for $\alpha>\alpha_o$.}
\end{definition}

\begin{definition} \label{defbranch} \rm
A set $K$ of solutions $(\alpha, w,x)$ to equation \eqref{eq:fde-reduced}
is called a {\it continuous branch of relative equilibria} bifurcating from the equilibrium $(\alpha_o,x_o)$
of equation \eqref{eq:fde1} if:
\begin{itemize}
\item[(i)] $x\not\in V^{\bold  S}$ for all $(\alpha, w,x)\in K$;

\item[(ii)] $\overline K$ contains a connected component $K_o$ such that $K_o\cap M\ne \emptyset$ (cf. \eqref{eq:trivial-HB});

\item[(iii)] For any $\varepsilon>0$ there is a $\delta>0$ such that if $(\alpha, w,x)\in K\cap K_o$
and $\|x\|<\delta$, then $|\alpha-\alpha_o|<\varepsilon$ and $|w-w_o|<\varepsilon$.
\end{itemize}
\end{definition}
\vs
A sufficient condition for the bifurcation of relative equilibria from the equilibrium $(\alpha_o,x_o)$, which provides an estimate for the number of possible branches of relative equilibria
with their symmetric properties, can be formulated as follows. 
\vs

\begin{proposition}\label{proposition}\label{prop-bifurcation-relative-equib}
Given system \eqref{eq:fde1}, assume conditions (A0)--(A2) are satisfied.  Let $(\mathcal H_o)$ be a maximal twisted orbit type in $V$.
Take decomposition \eqref{isot-Gamma-S1} and 
denote by $\mathfrak M$ the set of all $\mathcal G$-isotypical components 
in which $(\mathcal H_o)$ is an orbit type (cf. Remark \ref{rem-maximal-twisted-isot-comp}). Assume there exists $V_{j_o,k_o} \in \mathfrak M$ such that:

(i) $(\mathcal H_o)$ is a maximal twisted type in $V_{j_o,k_o} $;

(ii) $\mathfrak t_{j_o,k_o}(\alpha_o,w_o) \not=0 $ (cf. \eqref{crossingnumber});

(iii) ${\mathfrak t}_{j,k}(\alpha_o,w_o) \cdot  {\mathfrak t}_{j^{\prime},k^{\prime}}(\alpha_o,w_o) \geq 0$ \quad for all $V_{j,k}, V_{j^{\prime},k^{\prime}} \in \mathfrak M.$

\smallskip
\noindent
Then, there exist at least $|\mathcal G / \mathcal H_o|_{\bold S}$ continuous branches of relative equilibria of equation \eqref{eq:fde1}
bifurcating from the equilibrium $(\alpha_o,x_o)$ with the minimal symmetry $(\mathcal H_o)$ (here $|\cdot |_{\bold S}$ stands for the number of $\bold S$-orbits).
\end{proposition}

The proof literally follows the argument presented in \cite{AED}.
For completeness, here we give a brief sketch of the proof. Under extra transversality/genericity conditions,
this statement is well-known, see for example \cite{GS, GolSchSt}.

\paragraph{\bf Sketch of the proof of Proposition \ref{prop-bifurcation-relative-equib}.} The proof splits into three steps.

{\it (a) Auxiliary function and admissibility.} 
Take $\mathbb R^2$ with the trivial $\mathcal G$-action and define a $\mathcal G$-invariant neighborhood of the point $(\alpha_o,w_o,x_o)$ in $\br^2\oplus V$ of the form
\[
\Omega:=\{(\alpha,w,x): |\alpha-\alpha_o|<\ve,\, |w - w_o| <\ve, \, \|\overline x-x_o\|<\ve,\, \|x_*\|<\delta\},
\]
where $(\alpha,w,x)=(\alpha,w,\overline x+x_*)$, $(\alpha,w,\overline x)\in M$, $x_*\perp \overline x$,
and a $G$-invariant continuous function $\zeta: \overline{\Omega}\to \mathbb{R}$ satisfies
\begin{gather*}
\zeta(\alpha,w,x)<0 \;\;\text{ if }\; (\alpha,w,x)\in M\cap \overline{\Omega},\\
\zeta(\alpha,w,x)>0 \;\;\text{ if }\; (\alpha,w,x)\in \overline{\Omega} \;\text{ and } \|x_*\|=\delta
\end{gather*}
(
recall that $\zeta$ is called an {\it auxiliary function}). By condition (A0),
the map $\Phi_\zeta:\br^2\times V\to \br\times V$ defined by
\[
\Phi_\zeta(\alpha,w,x)=(\zeta(\alpha,w,x), \Phi(\alpha,w,x)), \quad (\alpha,w,x)\in \overline \Omega,
\]
is $\mathcal G$-equivariant (cf. \eqref{eq:bif-HB}). Moreover,
condition (A2) allows us to choose the parameters $\ve,\ \delta>0$ of the set $\Omega$ to be sufficiently small to ensure that
the map $\Phi_\zeta$ 
is $\Omega$-admissible (i.e., $\Phi_\zeta$ does not have zeroes on $\partial \Omega$).

\medskip
{\it (b) Twisted degree and a sufficient condition for the bifurcation of relative equilibria.}
Since $\Phi_\zeta$ is $\mathcal G$-equivariant and $\Omega$-admissible, the twisted degree 
\begin{equation}\label{eq:equiv-degree}
\maGdeg(\Phi_\zeta,\Omega) = \sum_{(\mathcal H)}n_{\mathcal H}(\mathcal H) 
\end{equation}
is correctly defined (here, $n_{\mathcal  H}\in \mathbb Z$ and the summation is going over all twisted orbit types occurring in $V$).
The following statement is parallel to Theorem 9.28 from \cite{AED}.

\begin{proposition}\label{prop:twisted-sufficient}
Given \eqref{eq:equiv-degree}, assume that $n_{\mathcal H_o} \not=0$ for some maximal twisted orbit type $(\mathcal H_o)$ in $V$. 
Then, the conclusion of Proposition \ref{prop-bifurcation-relative-equib} holds.
\end{proposition}

{\it (c) Twisted degree and crossing numbers.} To effectively apply Proposition \ref{prop:twisted-sufficient} for proving Proposition \ref{prop-bifurcation-relative-equib}, one 
needs to link the twisted degree  \eqref{eq:equiv-degree} to (isotypical) crossing numbers (cf. Definition \ref{isotyp-crossing-number}). To this end, one can use the following 
standard computational formula (cf. \cite{AED} and \eqref{isot-Gamma-S1}):
\begin{equation}\label{eq:deg-fde}
\maGdeg(\Phi_\zeta,\Omega)=
\prod_{\mu\in \sigma_-}\prod_{i=0}^r \Big(\deg_{\mathcal V_i}\Big)^{m_i(\mu)}\bullet
\sum_{j,k} {\mathfrak t}_{j,k}(\alpha_o,w_o)\, \deg_{\mathcal V_{j,k}},
\end{equation}
where $j=0,1,\dots,s$, $k = 1,...,m$; 
$\sigma_-$ denotes the set of all (real) negative roots  $\mu$ of the quasi-polynomial $\mathcal P(\alpha_o,\lambda)$ at $x_o$; $m_i(\mu)$ stands for the $\cV_i$-isotypical multiplicity of $\mu$;  $\deg_{\mathcal V_i}$ (resp. $\deg_{\mathcal V_{j,k}}$) denote the so-called basic degrees related to irreducible $\bold \Gamma$-representations (resp. 
$\mathcal G$-representations); and, ``$\bullet $" stands for the multiplication in the Euler ring $U(\mathcal G)$ (cf. \cite{tD}; 
see also Appendix for more details). Take $(\mathcal H_o)$  and  $V_{j_o,k_o} $ satisfying (i)--(iii). Then (see conditions (i)--(ii)), $(\mathcal H_o)$ appears with non-zero coefficient in  
$\mathfrak {t}_{j_o,k_o}(\alpha_o,w_o) \deg_{\mathcal V_{j_o,k_o}}$. Condition (iii) implies that $(\mathcal H_o)$ ``survives" in the sum given in the right hand side of \eqref{eq:deg-fde}. Finally,
since any $\deg_{\mathcal V_i}$ is an invertible element of the Burnside ring $A(\mathcal G) \subset U(\mathcal G)$, the result follows.

\vs

\subsection{Hopf bifurcation from a relative equilibrium}\label{sec:rel-eq}
\paragraph{\bf Regular relative equilibria.} Suppose that for some $\alpha=\overline \alpha$ and $\overline x \in V \setminus V^{\bold S}$,
equation \eqref{eq:fde1}  has a relative equilibrium 
$\bold S (\overline x)$  
(see Definition \ref{defrel}).
Then,  equation \eqref{eq:fde-reduced} is satisfied for some $\overline w\in\mathbb R$,
\begin{equation}\label{dd}
\Phi(\overline \alpha,\overline w, \overline x)=\wt f(\overline \alpha, i \overline w ,\overline x)-\overline  wJ \overline x=0,
\end{equation}
where $\overline x=(\overline q,\overline v)\in V^{\bold S}\oplus V_*$ with $\overline v\ne 0$.
Since
 \[
\left.\frac{d}{d\tau} \Phi(\overline \alpha,\overline w, e^{\tau J}\overline x)\right|_{\tau=0} =
 \left.D_x \Phi(\overline \alpha,\overline w, e^{\tau J}\overline x) J e^{\tau J}\overline x \right|_{\tau=0}=
D_x\Phi(\overline \alpha,\overline w, \overline x) J\overline x,
 \]
relation \eqref{dd} and the $\bold S$-equivariance of $\Phi$ imply
 that the directional derivative of $ \Phi$ at the point $\overline x$ in the direction of the orbit $\bold S (\overline x)$ is zero:
 \begin{equation}\label{degenerate}
 D_x\Phi(\overline \alpha,\overline w, \overline x) J\overline x=0.
 \end{equation}
 That is, the map $ D_x\Phi(\overline \alpha,\overline w, \overline x)$ has a non-zero kernel.
\vs
\begin{definition}\rm
A relative equilibrium
$\bold S( \overline x)$ will be called {\it regular} if the kernel of the map given by the following block-matrix
\begin{equation}\label{xxx}
\Big[ D_w \Phi(\overline \alpha,\overline w,\overline x)\;|\;  D_x \Phi(\overline \alpha,\overline w,\overline x)\Big] :\br\times V \to V
\end{equation}
is {\it one-dimensional}.
\end{definition}
\vs
We  make the following assumption:
\smallskip
\begin{itemize}
\item[(A3)]
Equation \eqref{eq:fde1} has a regular relative equilibrium $\bold S(\overline x)$
for some $\alpha=\overline\alpha$, $w=\overline w$.
\end{itemize}
\smallskip

The implicit function theorem implies that there exist a neighborhood $\mathscr U$ of $\overline \alpha $ in $\br$ and the functions $w:\mathscr U\to \br$ , $w(\overline \alpha)=\overline w$, and $u:\mathscr U\to S:=\{x\in V: x\bullet J\overline{x}=0\}$, $u(\overline \alpha)=0$, such that
\begin{equation}\label{eq:Phi-alpha}
\Phi(\alpha,w(\alpha),\overline x +u(\alpha))=0,\quad \alpha\in \mathscr U.
\end{equation}
Put $x(\alpha)=\overline x+u(\alpha)$, then we clearly have that
\begin{equation}\label{re}
 x_\alpha(t):=e^{ (w(\alpha) t +\tau)J} x(\alpha)
\end{equation}
is a branch of relative equilibria parametrized by $\alpha\in \mathscr U$.
It will be assumed that this branch has symmetric properties (cf.~Proposition \ref{proposition}):
\smallskip
\begin{itemize}
\item[(A4)]  The regular relative equilibrium $\bold S(\overline x)$ admits a twisted group symmetry $\mathcal H <  \mathcal G$.
\end{itemize}
\smallskip
 
Due to the equivariance, the twisted symmetry group $\mathcal H$ is the same for all relative equilibria $\bold S\big(x(\alpha)\big)$, $\alpha\in \mathscr U$,
of the branch \eqref{re}. We note that the relation
 \begin{equation}\label{degenerate'}
 D_x\Phi( \alpha, w(\alpha), x(\alpha)) J x(\alpha)=0,
 \end{equation}
 which is similar to \eqref{degenerate}, holds, and the relative equilibria $\bold S(x(\alpha))$
 are regular for all $\alpha\in\mathscr U$.

\paragraph{\bf Hopf bifurcation of relative periodic solutions.}
We are interested in finding  
solutions to \eqref{eq:fde1} of the form
\begin{equation}\label{eq:3}
x(t)=e^{(w(\alpha)+\phi)t J }(x(\alpha)+y(t)),
\end{equation}
where $y(t)$ is a non-stationary $p$-periodic function with $p=2\pi/\beta$  for some $\beta>0$,
and in symmetric properties of these solutions.  Here $y(t)$ and $\beta, \phi\in \br$ are unknown.
Periodic and quasi-periodic solutions of type \eqref{eq:3} are called {\it relative periodic} solutions.

To be more precise, let us define the   so-called {\it equivariant Hopf bifurcation} of small amplitude {\it relative periodic} solutions of type \eqref{eq:3}
from the family of relative equilibria \eqref{re}.
\vs

 The following definition is similar to Definition \ref{defbranch}.
\vs
 \begin{definition} \label{defbranch1} \rm
A set $K$ of quadruplets $(\alpha,\beta,\phi,x)$, where $x$ is a solution to equation \eqref{eq:fde1}
of the form \eqref{eq:3},
is called a {\it continuous branch of relative periodic solutions} bifurcating (via the {\it equivariant} Hopf bifurcation) from the relative equilibrium 
$(\overline \alpha,\overline w,\bold S(\overline x))$
 if there exists a $\beta_o>0$ such that:
\begin{itemize}

\item[(i)] $\overline K$ contains a connected component $K_o$ such that $ (\overline{\alpha},\beta_o,0,\overline{x})\in K_o$;

\item[(ii)] For any $\varepsilon>0$ there is a $\delta>0$ such that if $(\alpha, \beta,\phi,x)\in K\cap K_o$
and $\|y\|<\delta$, then $|\alpha-\overline\alpha|<\varepsilon$, $|\phi|<\delta$, and $|\beta-\beta_o|<\varepsilon$.

\end{itemize}
\end{definition}

 \vs
For a given $\alpha\in \mathscr U$, consider the derivative ${ D}_{\bold x} f(\alpha,x):  C_{-r}(V)\to  V$
of the functional $f$ with respect to ${\bold x}\in C_{-r}(V)$ and put
\begin{equation}\label{eq:B-alpha}
{\mathcal B}_\alpha:={D}_{\bold x} f(\alpha, e^{w(\alpha)J\cdot} x(\alpha))={D}_{\bold x} f(\alpha, \xi (i w(\alpha)) x(\alpha)): C_{-r}(V)\to V.
\end{equation}

For $\alpha\in  \mathscr U$ and $\lambda\in\mathbb C$, define the linear map ${\mathcal R}_\alpha : V^c\to V^c$
 in the complexification $V^c$ of $V$ by the formula
\begin{equation}\label{eq:R-alpha}
{\mathcal R}_\alpha(\lambda) y := {\mathcal B}_\alpha ( e^{(w(\alpha)J+\lambda\,\id)\cdot} y),\qquad y\in V^c.
\end{equation}
 Then, 
 \begin{equation}\label{charact'}
 {\rm det}_{\mathbb C}\, \left({\mathcal R}_\alpha(\lambda)-w(\alpha)J-\lambda\, \id\right)=0
 \end{equation}
 is the characteristic equation for the linearization of system \eqref{eq:fde1}
 on the relative equilibrium $\bold S (x(\alpha))$. Since
 \begin{equation}\label{eq:R(0)}
 {\mathcal R}_\alpha(0)-w(\alpha)J =D_x\Phi(\alpha,w(\alpha),x(\alpha)),
 \end{equation}
 equation \eqref{degenerate'} implies that the characteristic equation \eqref{charact'} has a zero root $\lambda=0$
 corresponding to the eigenvector $J x(\alpha)$;
 furthermore, due to the regularity of the relative equilibrium
 $\bold S(x(\alpha))$, this root is simple.

 \vs A necessary condition for the Hopf bifurcation is that
characteristic equation \eqref{charact'} has a pair of purely imaginary roots
$\lambda=\pm i \beta_o$, $\beta_o>0$, for $\alpha=\overline\alpha$.
We make a stronger assumption:

\begin{itemize}
\item[(A5)]  Characteristic equation \eqref{charact'} has a pair of purely imaginary roots
$\lambda=\pm i \beta_o$, $\beta_o>0$, for $\alpha=\overline\alpha$, and has no roots of the
form $\lambda=i\beta$, $\beta\ge0$, for $\alpha\ne \overline{\alpha}$, $\alpha\in \mathscr U$.
\end{itemize}

Put $\mathcal K:= \mathcal H \times S^1$ and consider the $\mathcal K$-isotypical decomposition of $V^c$:
\begin{equation}\label{decomp'}
V^c=U_{0,1}\oplus U_{1,1}\oplus \cdots \oplus U_{p,1},
\end{equation}
where $S^1$-action is given by complex multiplication.
Due to the equivariance, each isotypical component
$U_{j,1}$ is invariant for the map ${\mathcal R}_\alpha(\lambda)$ and for $J$. 
Therefore,
we can introduce the characteristic polynomial
\begin{equation}\label{barpi}
\overline{\mathcal P}_{j,1} (\alpha,\lambda):={\rm det}_{\mathbb C}\, \left(({\mathcal R}_\alpha(\lambda)-w(\alpha)J-\lambda\, \id )|_{U_{j,1}} \right),\qquad \lambda\in \mathbb C,
\end{equation}
associated with each isotypical component $U_{j,1}$, and define the $U_{j,1}$-isotypical crossing numbers
\begin{equation}\label{crossingnumber'}
\overline{\mathfrak t}_{j,1}(\overline{\alpha},\beta_o)=\bar{\mathfrak t}_{j,1}^-(\overline{\alpha},\beta_o)-\bar{\mathfrak t}^+_{j,1}(\overline{\alpha},\beta_o)
\end{equation}
at the point $(\overline\alpha,\beta_o)$ in the same way as we did in Subsection \ref{subsec-relative-equib} (cf.~\eqref{crossingnumber}). 
\vs

\begin{theorem}\label{t1}
Given system \eqref{eq:fde1}, assume conditions (A0) and (A3)--(A5) are satisfied.  
Take decomposition \eqref{decomp'} and let $(\mathcal L_o)$ be a maximal twisted orbit type in $V^c$.
Denote by $\mathfrak N$ the set of all $\mathcal K$-isotypical components in \eqref{decomp'} 
in which $(\mathcal L_o)$ is an orbit type. Assume there exists $U_{j_o,1} \in \mathfrak N$ such that:

(i) $(\mathcal L_o)$ is a maximal twisted orbit type in $U_{j_o,1} $ (cf. Remark \ref{rem-maximal-twisted-isot-comp});

(ii) $\mathfrak t_{j_o,1}(\alpha_o,w_o) \not=0 $ (cf. \eqref{crossingnumber});

(iii) ${\mathfrak t}_{j,1}(\alpha_o,w_o) \cdot  {\mathfrak t}_{j^{\prime},1}(\alpha_o,w_o) \geq 0$ \quad for all $U_{j,1}, U_{j^{\prime},1} \in \mathfrak N.$

\smallskip
\noindent
Then, there exist at least $|\mathcal H / \mathcal L_o|_{S^1}$ continuous branches of relative periodic 
solutions \eqref{eq:3} bifurcating via the Hopf bifurcation from the relative equilibrium 
$(\overline \alpha,\overline w,\bold S (\overline x))$  (cf. Definition \ref{defbranch1}) and having the minimal symmetry $(\mathcal L_o)$.
\end{theorem}

\subsection{Proof of Theorem \ref{t1}}

For the proof, which splits into several steps, we {modify} the twisted equivariant degree approach described in Sections 10.1-2 of \cite{AED} 
(see also the sketch of the proof of Proposition \ref{prop-bifurcation-relative-equib}). 

\medskip

{\it (a) Rescaling time.} 
Substituting (\ref{eq:3}) in   \eqref{eq:fde1} (see also \eqref{eq:Phi-alpha} and \eqref{re}) leads to  equations
\begin{equation}\label{eq:sys-HB}
\begin{cases}
 \dot y(t)=f(\alpha, \wt x+\wt y_t)-(w(\alpha)+\phi) J (x(\alpha)+y(t)),\\
y(t)=y(t+p),\end{cases}
\end{equation}
where $p>0$ is an unknown period of $y$ 
and
\begin{equation}\label{eq:4'}
\wt x(\theta):=e^{(w(\alpha)+\phi)\theta J}x(\alpha),\qquad \wt y_t(\theta):= e^{(w(\alpha)+\phi)\theta J}y(t+\theta).
\end{equation}
By normalizing the period $p={2\pi}/\beta$ of $y$, we obtain
the system
\begin{equation}\label{eq:4}
\begin{cases}
\dot y(t)=\frac1\beta \Big( f(\alpha, \wt x+\wt y^\beta_{ t})- (w(\alpha)+\phi) J(x(\alpha)+ y(t))\Big),\\
y(t)=y(t+2\pi)\end{cases}
\end{equation}
with 
\begin{equation}\label{eq:4''}
 \wt y_t^\beta(\theta):= e^{(w(\alpha)+\phi)\theta J}y(t+\beta\theta).
\end{equation}

\medskip
{\it (b)  Constraint.} This step reflects the specifics  of the Hopf bifurcation of {\it relative} periodic solutions from a {\it relative} equilibrium. Namely,
in order to ensure that the unknown function $y(t)$ is determined ``uniquely" (i.e.~up to shifting the argument), we will assume that this function 
satisfies an additional constraint.
From assumption (A3) and \eqref{degenerate}--\eqref{eq:Phi-alpha}, it follows that for any $\alpha\in \mathscr{U}$, the map given by the matrix
$D_x \Phi(\alpha,w(\alpha),x(\alpha))$ has the one-dimensional kernel ${\rm span}\,\{J x(\alpha)\}$.
Denote by $g^\dagger(\alpha)$ the adjoint eigenvector of the transpose matrix
$D_x \Phi(\alpha,w(\alpha),x(\alpha))^T$ corresponding to the zero eigenvalue:
\[
D_x \Phi(\alpha,w(\alpha),x(\alpha))^T g^\dagger (\alpha)=0,\qquad  g^\dagger(\alpha)\bullet J x(\alpha)=1,\quad \alpha\in \mathscr{U}.
\]
We will look for a solution to 
{\eqref{eq:4}} with the $y$-component satisfying
the constraint
\begin{equation}\label{eq:orth-con}
\mathscr J_\alpha(y) := g^\dagger(\alpha)\bullet \int_0^{2\pi}  y(t)dt = 0.
\end{equation}

{\it (c) Setting system \eqref{eq:4} in functional spaces.} Using the standard identification of
a $2\pi$-periodic $V$-valued function with the $V$-valued function on $S^1$, we reformulate system (\ref{eq:4}) with constraint \eqref{eq:orth-con} as a
$\mathcal K$-equivariant operator equation in the space $\br_+^2 \times \mathscr{W}$, where $\mathcal K$ acts trivially on  $\br_+^2:=\br\times \br_+$ 
and $\mathscr{W}:=H^1(S^1;V)$ stands for the first Sobolev space equipped with the $\mathcal K$-action given by 
\begin{equation}\label{eq:H-action}
(h,e^{i\tau})(u)(t):= h u(t + \tau) \quad  ((h,e^{i\tau}) \in \mathcal H \times S^1 =: \mathcal K, \; u \in  \mathscr{W}).  
\end{equation}
To this end, denote
\begin{equation}\label{eq:v-alpha}
v_\alpha:=D_w \Phi(\alpha,w(\alpha),x(\alpha))\in V 
\end{equation}
and observe that 
\begin{equation}\label{valpha}
v_\alpha\in V^{\mathcal H}.
\end{equation}
Indeed, the $\mathcal H$-action on $V$ induces the $\mathcal H$-action on $\mathbb{R}\times V$, where $\mathcal H$ acts trivially on $\mathbb{R}$.
Since the map $\Phi(\alpha,\cdot,\cdot): \mathbb{R}\times V\to V$ is $\mathcal H$-equivariant and $(w(\alpha),x(\alpha))\in (\mathbb{R}\oplus V)^{\mathcal H}$,
one has that $D \Phi(\alpha,w(\alpha),x(\alpha)): \mathbb{R}\times V\to V$ is $\mathcal H$-equivariant as well, which implies \eqref{valpha}.

Next, given an $\alpha\in \mathscr{U}$, we identify a function $z\in \mathscr W$ with the pair $(y,\phi)$, where $y\in \mathscr{W}$ satisfies \eqref{eq:orth-con} and $\phi\in\mathbb{R}$,
by the relationships
\begin{equation}\label{eq:proj1}
z= \phi\, v_\alpha + y,\qquad \mathscr J_\alpha(y)=0,
\end{equation}
and define the corresponding projections
\begin{equation}\label{proj}
\phi=\hat\pi_\alpha (z), \qquad y= z- \hat\pi_\alpha (z) v_\alpha.
\end{equation}
Let us introduce the following operators:
\begin{align*}
L& : \mathscr{W}  \to L^2(S^1;V),\quad & L(z)&= \dot z,\\
j& :   \mathscr{W} \to C(S^1;V),\quad \;\;\; &j(z)&=z,
\end{align*}
where 
$C(S^1;V)$ is the space of continuous functions 
equipped with the usual sup-norm. 
 Furthermore, define 
$F:\mathbb R^2_+\times C(S^1;V)\to V$ by
\begin{align}\label{eq:def-F}
F(\alpha,\beta,z(t))&: =\frac1\beta \Big( f(\alpha, \wt x+\wt y^\beta_{ t})- (w(\alpha)+\phi) J(x(\alpha)+ y(t))\Big), \quad t\in\br,
\end{align}
with $ (\alpha,\beta,z)\in \mathbb R^2_+$, $z\in C(S^1;V)$, where
the function $y\in C(S^1;V)$
and the scalar $\phi$ are defined by \eqref{proj}; $ \wt x, \wt y^\beta_{ t}$ are defined in \eqref{eq:4'}, \eqref{eq:4''}.
Next,  denote by ${\mathcal N}_{F}:\mathbb R^2_+\times  C(S^1,V)\to L^2(S^1;V)$ the Nemytsky operator associated with the map $F$, i.e.
\begin{equation}\label{eq:ode-defNf}
\Big({\mathcal N}_{F}(\alpha, \beta,z)\Big)(t):= F(\alpha,\beta,z(t)), \quad  z\in C(S^1;V).\end{equation}

Since $Lz=Ly$, system (\ref{eq:4}) with constraint \eqref{eq:orth-con} is equivalent to the following operator equation:
\begin{equation}\label{eq:funct-1}
L z=\mathcal N_F(\alpha,\beta,j(z)), \quad (\alpha,\beta)\in \br^2_+, \;\; z\in \mathscr W.
\end{equation}
Using the formulas similar to \eqref{eq:H-action}, one can define the $\mathcal H$-actions on an  $C(S^1,V)$ and $L^2(S^1;V)$.
Clearly, all the operators involved in formula \eqref{eq:funct-1} are $\mathcal K$-equivariant, therefore
equation (\ref{eq:funct-1}) can be transformed to a  
$\mathcal K$-equivariant fixed-point problem in $\br^2_+\times \mathscr W$ as follows. Define the operator $K: \mathscr{W}\to L^2(S^1;V)$ by
\begin{equation}\label{eq:ode-defK}
K(z):= \frac 1{2\pi}\int_0^{2\pi} z(t) \,dt, \quad z\in \mathscr{W},\end{equation}
which is simply a projection on the subspace $V$ of constant functions. Then,  the operator $L+K:\mathscr W\to L^2(S^1;V)$
is an isomorphism. Put
\begin{align}\label{eq:ode-F-def1}
\mathcal F(\alpha,\beta,z)&:=(L+ K)^{-1}\left[\mathcal  N_F(\alpha,\beta,j(z))+ K(z)\right],\\
\mathfrak F(\alpha,\beta,z)&:=z-\mathcal F(\alpha,\beta,z).\label{eq:ode-F-def2}
\end{align}
In this way, the following equation is equivalent to (\ref{eq:funct-1}):
\begin{equation}\label{eq:funct-2}
\mathfrak F(\alpha,\beta,z)=0,\qquad (\alpha,\beta,z)\in \br_+^2\times \mathscr W.
\end{equation}
\vs

{\it (d) Reduction to twisted degree.} Take 
$\overline \alpha, \mathscr{U}$ and  $\bold {S}(x(\alpha))$ provided by condition (A5) (see
also \eqref{eq:Phi-alpha}--\eqref{degenerate'}) and $\beta_o$ provided by (A5). 
Put
\begin{equation*} 
M:= \{(\alpha,\beta,z)\, : \, \alpha \in  \mathscr{U}, \beta \in \mathbb R_+, z \in \bold{S}(x(\alpha))\} \subset \mathbb R^2_+ \times \mathscr{W}, 
\end{equation*}
where 
\begin{equation}\label{eq:decomp-Sobolev}
\mathscr{W} = V \oplus \overline{\bigoplus_{l=1}^{\infty}\mathscr{W}_l}, 
\quad \mathscr{W}_l = \{ e^{i l t} \cdot y_l \; : y_l \in V^c\},
\end{equation}
and the subspace of constant functions is identified with the space $V$. For any small $\varepsilon > 0$, define a three-dimensional $\mathcal K$-invariant 
submanifold
$$
M_{\varepsilon} := \{(\alpha,\beta,z) \in M\, : \; |\alpha - \overline{\alpha}| < \varepsilon, \; |\beta - \beta_o| < \varepsilon\} \subset 
\mathbb R^2_+ \times V \subset \mathbb R^2_+ \times \mathscr{W}
$$
of $M$.
Take a small $r >0$,  define a normal $\mathcal K$-invariant neighborhood of  $M_{\varepsilon}$ by
$$
\mathscr{N}_{\varepsilon,r} := \{u+v \in \mathbb R^2_+ \times  \mathscr{W} \; : \; u \in M_{\varepsilon}, \; v \perp \tau_u(M_{\varepsilon}), \; \|v\| < r\}
$$
and denote
$$
\partial _M^{\mathscr{N}}:=   \partial (\mathscr{N}_{\varepsilon,r}) \cap M,  \quad    \partial _r^{\mathscr{N}}:= \{u +v \in  \mathscr{N}_{\varepsilon,r} \; : \; 
\|v\| = r\}. 
$$
By condition (A5), one can choose $\varepsilon$ and $r$ to be so small that 
$$
\mathfrak F^{-1}(0) \cap \partial  (\mathscr{N}_{\varepsilon,r}) \subset \partial _M^{\mathscr{N}} \cup  \partial _r^{\mathscr{N}}.
$$
Let $\xi : \overline{\mathscr{N}_{\varepsilon,r}} \to \mathbb R$ be a $\mathcal K$-invariant Urysohn function which is positive on  $\partial _r^{\mathscr{N}}$
and negative on  $\partial _M^{\mathscr{N}}$. Then, the map $\mathfrak F_{\xi} : \overline{\mathscr{N}_{\varepsilon,r}} \subset \mathbb R^2_+ \times \mathscr{W} \to \mathbb R \times \mathscr{W}$ given by
$$
\mathfrak F_{\xi}(\alpha,\beta,z): = (\xi(\alpha,\beta,z), \mathfrak F(\alpha,\beta,z))
$$  
is $\mathcal K$-equivariant and $\mathscr{N}_{\varepsilon,r}$-admissible, therefore the $\mathcal K$-equivariant twisted degree 
\begin{equation}\label{eq:twisted-deg-K}
\mathcal K{\rm -deg}(\mathfrak F_{\xi},\mathscr{N}_{\varepsilon,r}) = \sum_{(\mathcal L)}n_{\mathcal L}(\mathcal L)
\end{equation}  is correctly defined. 

\vs
\begin{proposition}\label{prop:sufficient-relative-periodic}
Let $(\mathcal L_o)$ and   $(\overline \alpha,\overline w,\bold S (\overline x))$  be as in Theorem \ref{t1} and assume
that  $n_{\mathcal L_o} \not=0$ in \eqref{eq:twisted-deg-K}.  Then, the conclusion of Theorem \ref{t1} holds.
\end{proposition}
\begin{proof}
Following the same argument as in the proof of Theorem 9.28 from \cite{AED}, one can  establish the existence of a continuous branch of solutions 
$(\alpha,\beta,z)$ to equation \eqref{eq:funct-2} bifurcating from $(\overline{\alpha}, \beta_o,\overline{x})$ with symmetry $(\mathcal L_o)$. 
For any solution $(\alpha,\beta,z)$ belonging to this branch, take $v_{\alpha}$ given by \eqref{eq:v-alpha} and identify $\phi$ and $y$ using \eqref{eq:proj1} and 
\eqref{proj}. Then, the quadruplets $(\alpha,\beta,\phi,y)$ constitute a continuous branch required in the conclusion of Theorem \ref{t1}. Symmetric properties of this
branch are guaranteed by condition \eqref{valpha} and assumption $n_{\mathcal L_o} \not=0$. 
\end{proof}

{\it (e) Computation of twisted degree.} 
To effectively apply Proposition \ref{prop:sufficient-relative-periodic} to proving Theorem \ref{t1}, one 
needs to prove that the hypotheses of Theorem \ref{t1} indeed guarantee a non-zero summand $n_{\mathcal L_o}(\mathcal L_o)$ in twisted degree \eqref{eq:twisted-deg-K}. 
To estimate \eqref{eq:twisted-deg-K}, one can use a computational product formula similar to \eqref{eq:deg-fde} (cf. \cite{AED}). To this end, one needs: 

\smallskip

(i) to show that the restriction of $D_z\mathfrak F(\alpha,\beta,\overline{x})$ to $V$ is invertible;

\smallskip

(ii) to link the restriction of $D_z\mathfrak F(\alpha,\beta,\overline{x})$ to  $\overline{\bigoplus_{l=1}^{\infty}\mathscr{W}_l}$ to crossing numbers.

\smallskip
\noindent
Both problems require to evaluate 
the linearization of $F$ (cf.  \eqref{eq:sys-HB}--\eqref{eq:4''} and \eqref{eq:def-F}--\eqref{eq:ode-F-def2}). Assuming in \eqref{eq:def-F} $\phi$ and $y$ to be small, one obtains for the first summand 
(up to the {\it terms of higher order}):
\begin{equation}\label{eq:main1}
\begin{split}
 f(\alpha, \wt x+\wt y^\beta_{ t}) 
 &= f \big(e^{(w(\alpha) + \phi)\theta J}(x(\alpha) + y(t +\beta\theta)\big) = f \big(e^{(w(\alpha) \theta J} x(\alpha)\big) \\
 &+ D_xf(e^{w(\alpha)J\theta} x(\alpha))  \Big[e^{(w(\alpha) + \phi) \theta J} (x(\alpha) + y(t + \beta \theta)) -   e^{w(\alpha) \theta J}  x(\alpha)  \Big]\\
 &+ (t.o.h.o.).
\end{split}
\end{equation}
The expression in square brackets reads:
\begin{equation}\label{eq:main2}
\begin{split}
e^{w(\alpha) \theta J} e^{\phi \theta J} x(\alpha)  
&+ e^{(w(\alpha) + \phi)\theta J} y(t + \beta \theta) - e^{w(\alpha) \theta J}  x(\alpha) \\ 
&=  e^{w(\alpha) \theta J} \big(e^{\phi \theta J} - {\rm Id}\big) x(\alpha)  + e^{(w(\alpha) + \phi)\theta J} y(t + \beta \theta) \\
&= \phi\theta J e^{w(\alpha) \theta J} x(\alpha)  +  e^{w(\alpha)\theta J} ({\rm Id} + \phi\theta J)y(t + \beta \theta)\\
&= \phi\theta J e^{w(\alpha) \theta J} x(\alpha)  +  e^{w(\alpha)\theta J}y(t + \beta \theta).
\end{split}
\end{equation}
Combining \eqref{eq:main1} and \eqref{eq:main2} yields
\begin{equation}\label{eq:main3}
\begin{split}
 f(\alpha, \wt x+\wt y^\beta_{ t}) &=   f \big(e^{(w(\alpha) \theta J} x(\alpha)\big) \\
&+ D_xf(e^{w(\alpha)J\theta} x(\alpha)) \big(\phi\theta J e^{w(\alpha) \theta J} x(\alpha)  +  e^{w(\alpha)\theta J}y(t + \beta \theta)\big) + (t.o.h.o.).
\end{split}
\end{equation}
The linearization of other summands in \eqref{eq:def-F} gives:
\begin{equation}\label{eq:main4}
- (w(\alpha)+\phi) J(x(\alpha)+ y(t)) = -w(\alpha) J y - \phi J x(\alpha) +  (t.o.h.o.).
\end{equation}
Combining now \eqref{eq:main3}, \eqref{eq:main4},  \eqref{eq:def-F}  with \eqref{eq:v-alpha} and \eqref{eq:bif-HB} yields the following formula for the linearization
of $F$:
\begin{equation}\label{deriv}
D_{\bold z} F(\alpha,\beta,e^{w(\alpha)\theta J}x(\alpha))=\frac1\beta\left( \phi\,v_\alpha +
D_{\bold x}f(\alpha,e^{w(\alpha)\theta J}x(\alpha))\, e^{w(\alpha) \theta J} y(t+\beta\theta)-w(\alpha)J y(t)
\right),
\end{equation}
where $y$ and $\phi$ are defined by \eqref{proj}.
Therefore, $D_z\mathfrak F(\alpha,\beta,\overline{x})_{|V}$ 
has the form
\begin{equation}\label{constlin}
D_z\mathfrak F(\alpha,\beta,\overline{x})z = \phi\, D_w \Phi(\alpha,w(\alpha),x(\alpha)) + D_x \Phi(\alpha,w(\alpha),x(\alpha)) y_0,
\end{equation}
where $\phi=\hat\pi_\alpha(z)\in\mathbb{R}$ and $y_0=K(z)\in V$ satisfies $g^\dagger(\alpha)\bullet y_0=0$. Due to \eqref{constlin},
from assumption (A3) (see \eqref{xxx}), one obtains that
$D_z\mathfrak F(\alpha,\beta,\overline{x})_{|V}$ 
is invertible in a neighborhood
of the point $\alpha=\overline\alpha$.
Therefore (cf. Step (c) of the proof of Proposition \ref{prop-bifurcation-relative-equib}),  $D_z\mathfrak F(\alpha,\beta,\overline{x})_{|V}$ 
does not affect the existence of maximal twisted orbit types in \eqref{eq:twisted-deg-K} and, therefore, 
is of no consequence for the analysis of maximal twisted orbit types of relative periodic solutions.

On the other hand,  $D_z\mathfrak F(\alpha,\beta,\overline{x})_{|\mathscr W_l}$ acts as follows (cf. \eqref{eq:decomp-Sobolev}):
\begin{equation}\label{prb}
  D_z\mathfrak F(\alpha,\beta,\overline{x})y_l = D_{\bold x}f(\alpha,e^{w(\alpha) \theta J}x(\alpha))\, e^{(w(\alpha)J+i \beta l \,{\rm Id})\theta}y_l - \left(w(\alpha)J + i\beta l\,{\rm Id}\right)y_l.
\end{equation}
Also, since $ D_z\mathfrak F(\alpha,\beta,\overline{x})$ is $\mathcal K$-equivariant, it preserves $\mathcal K$-isotypical decompositions of $\mathscr W_l$ for all $l$.
Take $l = 1$ and consider decomposition \eqref{decomp'}. For the restriction  $D_z\mathfrak F(\alpha,\beta,\overline{x})_{|U_{j,1}}$, one has:
$$
\Delta_{j}(\alpha,\beta):= {\rm det}_{\mathbb C}\, D_z\mathfrak F(\alpha,\beta,\overline{x})_{|U_{j,1}} = \overline{\mathcal P}_{j,1} (\alpha,i\beta).
$$
Therefore, the degree of the planar vector field $\Delta_j$ equals the crossing number
\eqref{crossingnumber'}. 

Applying the same argument as in Step (c) of the proof of Proposition \ref{prop-bifurcation-relative-equib} completes the proof of Theorem \ref{t1}.

\vs


\section{DDE Model of a Symmetric Configuration of Passively Mode-Locked Semiconductor 
	Lasers}
\subsection{Mathematical model}
In \cite{laser}, a model for a mode-locked semiconductor laser with gain and absorber sections was introduced as a system of the following delay differential equations:
\begin{equation}\label{eq:ML-a}
\begin{cases}
\dot g(t) =g_0-\gamma_g g(t) -\frac 1{E_g} e^{-q(t)}(e^{g(t)}-1)|a(t)|^2, \\
\dot q(t)=q_0-\gamma_qq(t)-\frac 1{E_q}\left( 1-e^{-q(t)}  \right)|a(t)|^2,\\
\dot a(t)=-\gamma a(t)+\gamma\sqrt \kappa \exp\left[\frac{(1-i  \eta_g)g(t-T)- (1-i  \eta_q)q(t-T)}2\right]a(t-T).
\end{cases}
\end{equation}
The complex-valued function $a(t)$ is the field amplitude at the entrance of the absorber section with $|a(t)|^2$ representing the optical power. The real-valued functions $g(t)$ and $q(t)$  represent saturable gain and losses, respectively, and $  \eta_g$, $  \eta_q$ are the linewidth enhancement factors corresponding to self-phase modulation. The constants $g_0$ and $q_0$ stand for unsaturated gain and absorption. The constants 
$\gamma_g$ and $\gamma_q$ are the carrier density relaxation rates in the gain and absorbing sections;
$E_g$ and $E_q$ are the saturation energies in the these sections; the ratio  $s={E_g}/{E_q}$ is important 
for laser dynamics. Finally, $T$ stands for the cold cavity round-trip time, and $\sqrt \kappa$ is the linear non-resonant attenuation factor per pass. 
The parameter $g_0$ is proportional to the pump current, which is the physical control parameter.

Assume $
(g(t),q(t),a(t))^\top\in \mathbb R\oplus \br\oplus \bc\simeq \br^4 =: \mathscr V$ and equip  $\mathscr V$  with the natural $S^1$-representation 
(trivial on $(g,q)$-components and complex multiplication on $a$-component). 
Clearly, system
\eqref{eq:ML-a} is $S^1$-equivariant.
In what follows, assuming  the value $\alpha:=g_o$ to be the bifurcation parameter, we will show how Proposition \ref{prop-bifurcation-relative-equib}  
(resp. Theorem \ref{t1}) can be used to study bifurcations of relative equilibria (resp. relative periodic solutions) for the 
network of identical oscillators \eqref{eq:ML-a} coupled in a $D_n$-symmetric fashion.  

\subsection{$D_n$-configuration of identical semiconductor lasers}\label{subsect-D-n-configuration} 
Let  $\mathfrak f : \mathbb R\times C([-T,0];\mathscr V)\to\mathscr  V$ be the map induced by the right-hand side of system \eqref{eq:ML-a}.
Put $ V:= \mathscr  V^n$ and define the map $f_o:\br\oplus C([-T,0];V)\to V$ by
\begin{equation}\label{eq:f_o}
f_o(\alpha, x_t)=\Big(\mathfrak f(\alpha,x^0_t),\mathfrak f(\alpha,x^1_t),\dots,\mathfrak f(\alpha,x^{n-1}_t)\Big)^\top,
\end{equation}
where $x=(x^0,x^1,\dots,x^{n-1})^\top \in V$. Take the linear operator $C:\mathscr V\to \mathscr V$ with the matrix
 \begin{equation}\label{eq:matrix-C}
 C:=\left[\begin{array}{ccc}0&0&0\\
 0&0&0\\
 0&0&e^{i\psi}   \end{array}  \right]
 \end{equation}
and let $\mathscr C:V\to V$ be given by the block matrix
\begin{equation}\label{eq:C}
\mathscr C:=\left[\begin{array}{cccccc}0&C&0&\dots&0&C\\
C&0&C&\dots&0&0\\
0&C&0&\dots&0&0\\
\vdots&\vdots&\vdots&\ddots&\vdots&\vdots\\
0&0&0&\dots&0&C\\
C&0&0&\dots&C&0
    \end{array}   \right].
\end{equation}
We are interested in solutions to the system 
\begin{equation}\label{eq:f-Dn} 
\dot{x} = f(\alpha,\eta, x,x_t):= f_o(\alpha,x_t)+\eta \mathscr C x,\quad x\in V,\; \alpha,\,\eta\in \br,
\end{equation}
where $\eta$ stands for the strength of coupling.

Clearly, the space $V$ is an orthogonal $D_n\times S^1$-representation for $\Gamma:=D_n$, where $D_n$-action on $V$ is defined by permutation of the coordinates of the vector $x\in V$. More precisely, $D_n$ stands for the dihedral group being  the group of symmetries of a regular $n$-gone, i.e. one can consider it to be a subgroup of the symmetric group $S_n$ of $n$-vertices $\{0,1,\dots,n-2,n-1\}$  of the regular $n$-gone. This group is generated by
the ``rotation'' $\xi:= (0,n-1,n-2,\dots,1)$ and the ``reflection'' $\boldsymbol\kappa:=(1,n-1)(2,n-2)\dots(m,n-m)$, where $m=\left\lfloor \frac{n-1}2\right\rfloor$. Then, the $D_n\times S^1$-action on  $V$ is given by
\begin{equation}\label{eq:act1}
(h, e^{i\tau})x= (e^{i\tau}x^{h(0)},e^{i\tau}x^{h(1)},\dots,e^{i\tau}x^{h(n-1)})^{\top}, \quad e^{i\tau}\in S^1,\, h \in D_n,
\end{equation}
where $x=(x^0,x^1,\dots,x^{n-1})^\top\in V$ and $e^{i\tau}$ acts on $x^i \in \mathscr  V \simeq \mathbb R \oplus \mathbb R \oplus \mathbb C$ trivially on the first two components and by complex multiplication on the $\mathbb C$-component ($i = 0,...,n-1$). Obviously, system \eqref{eq:f-Dn} satisfies condition (A0). 
\begin{remark}\label{rem:symmetries-relative-equilib}
{\rm Recall, if ${\bf S}(\overline{x})$ is a relative equilibrium for system  \eqref{eq:f-Dn}, then symmetries of ${\bf S}(\overline{x})$ are completely determined by a (twisted) isotropy subgroup $\mathcal G_{\overline{x}}$ with respect to the $\mathcal G : = D_n \times S^1$-action.
} 
\end{remark}

Observe that $x_o(\alpha):=  \big(\frac{\alpha}{\gamma_g}, \frac{q_0}{\gamma_q},0\big) \in \mathscr  V$ is an equilibrium of system \eqref{eq:ML-a} for any $\alpha$, hence 
\[
\mathcal O(\alpha):= (x_o(\alpha),x_o(\alpha),...,x_o(\alpha)) \in V
\] is an equilibrium of system \eqref{eq:f-Dn} for any $\alpha$. Also
(cf. \eqref{eq:linearization-matrix}),
\begin{equation}\label{eq:linearization-at-x0}
{D}_{\bold x} \mathfrak f(\alpha, x_o(\alpha))_{|{\mathscr  V}} = \left[\begin{array}{ccc}   
-\gamma_g &0&0\\
0&-\gamma_q&0\\
0&0&\left(\sqrt \kappa \exp\left[\frac{(1-i  \eta_g)\frac{\alpha}{\gamma_g} - (1-i  \eta_q) \frac{q_0}{\gamma_q} }2\right]-1\right)\gamma
\end{array}   \right].
\end{equation}

\subsection{Bifurcation of symmetric relative equilibria}\label{subsec:bif-rel-equil} 

{Hereafter,} for the sake of simplicity, we will restrict ourselves to the case $n = 8$.

\medskip
\noindent
{\bf Isotypical decomposition and maximal twisted orbit types.} 
To apply Proposition  \ref{prop-bifurcation-relative-equib} for studying relative equilibria bifurcating from the the equilibrium $\mathcal O(\alpha)$, observe that $V$ admits 
the isotypical $D_8$-decomposition:
\begin{equation}\label{eq-D_n-decomp}
V = \bigoplus_{j=0}^4 W_j,
\end{equation} 
where $W_j$ is modeled on $\mathcal V_j$,\ $\mathcal V_1$ is a one-dimensional trivial representation, $\mathcal V_4$ is a one-dimensional $D_8/D_4$-representation and $\{\mathcal V_j\}_{j=1}^3$ are three two-dimensional non-equivalent irreducible representations with different actions of the rotational generator (see \cite{AED} for details). Observe also (see \eqref{isot-Gamma-S1}) that decomposition \eqref{eq-D_n-decomp} can be refined to the
$D_8 \times S^1$-decomposition: $W_j = V_j^0 \oplus V_{j,1} $, $j = 0,...,4$, where $V_j^0$ is modeled on $\mathcal V_j \simeq \mathbb R^2$ and $S^1$ acts trivially, while $V_{j,1}$
is modeled on $\mathcal V_{j,1} \simeq \mathbb C^2$ and $S^1$ acts by complex multiplication (see \cite{AED}). Clearly, $\dim W_0 = \dim W_4 = 4$, while $\dim W_1 = \dim W_2 = \dim W_3 = 8$. 

Let us now describe maximal twisted orbit types in $V$. By inspection, for any $j = 0,1,2,3,4$, if $(\mathcal H_o)$ is a maximal orbit type in $V_{j,1}$, then  
$(\mathcal H_o)$ is a maximal twisted type in $V$ (cf. Proposition \ref{prop-bifurcation-relative-equib}, assumption (i)). In turn, the list of maximal twisted types in any $V_{j,1}$ is given by 

\smallskip

for $V_{0,1}$: \quad $(D_8 \times \{1\}) \simeq (D_8)$;

\smallskip

for $V_{1,1}$: \quad $(\mathbb Z_8^{t_1}), (D_2^d), (\widetilde{D}_2^d)$; 

\smallskip

for $V_{2,1}$: \quad $(\mathbb Z_8^{t_2}), (D_4^d), (\widetilde{D}_4^d)$; 

\smallskip

for $V_{3,1}$: \quad $(\mathbb Z_8^{t_3}), (D_2^d), (\widetilde{D}_2^d)$; 

\smallskip

for $V_{4,1}$: \quad $(D_8^d)$.

\smallskip\noindent
We refer to Subsection \ref{notation:0} the Appendix for explicit description of all these subgroups,
see also Remark \ref{rem:symmetries-relative-equilib}.

\medskip
\noindent
{\bf Equivariant spectral reduction and condition (A1).}
The linearization ${D}_{\bold x} f_o(\alpha,{\bold x}): \br\times C_{-r}(V)_{|V} \to V$ of system \eqref{eq:f-Dn} at  $\mathcal O(\alpha)$ 
respects isotypical decomposition 
\eqref{eq-D_n-decomp}. To describe its action on isotypical components, 
define a (real) $4 \times 4$-matrix $\xi$ by
\begin{equation}\label{eq:matrix-ksi}
 \xi :=\left[\begin{array}{ccc}1&0&0\\
 0&1&0\\
 0&0&e^{i{2\pi \over 8}}   \end{array}  \right]
\end{equation}
and put 
\begin{equation}\label{eq:A-i}
A_j := {D}_{\bold x} \mathfrak f(\alpha, x_o(\alpha))_{| \mathcal V} +  \eta C (\xi^j + \xi^{-j}), \quad\quad  j = 0,1,2,3,4
\end{equation}
(cf. \eqref{eq:linearization-at-x0}
and \eqref{eq:matrix-C}; it should be stressed that $A_j$ is considered here as a {\it real} $4 \times 4$-matrix). 
Then,  
\begin{equation}\label{eq:lineariza-f_o}
{D}_{\bold x} f_o(\alpha, \mathcal O(\alpha))_{| W_j}   = 
 \begin{cases}
 A_j \quad\quad\quad\quad\quad \;\, {\rm if} \; j = 0,4,\\
 \left[\begin{array}{cc}A_j & 0\\
 0 & A_j\\
 \end{array}  \right] \quad {\rm if}\; j = 1,2,3.
 \end{cases} 
\end{equation}
Since the action of $S^1$ on $(g,q)$-components of \eqref{eq:ML-a} is trivial, it follows from \eqref{eq:linearization-at-x0} and \eqref{eq:matrix-ksi}-\eqref{eq:lineariza-f_o} that $\det({D}_{\bold x} f_o(\alpha, \mathcal O(\alpha))_{| V^{S^1}}) = (\gamma_g \gamma_q)^8 \not = 0$, hence 
(see \eqref{eq:quasipolynomials}), $\mathcal P_0(\alpha,0, \mathcal O(\alpha)) \not= 0$ so that system \eqref{eq:f-Dn} satisfies condition (A1).  

\medskip
\noindent
{\bf Characteristic quasi-polynomial.}
Next, let us consider the characteristic quasi-polynomial  $\mathcal P_{\ast}(\alpha,\lambda, \mathcal O(\alpha))$ (see \eqref{eq:quasipolynomials}).  
For any $j = 0,1,2,3,4$, put
\begin{equation}\label{eq:mm}
\widetilde{\mathcal P_j} := \lambda +  \gamma - \gamma \sqrt \kappa \exp\left[\left( \frac{\alpha}{2\gamma_g} - \frac{q_0}{2 \gamma_q}\right) 
+ i \left( \frac{\eta_q q_0}{2 \gamma_q} - \frac{\eta_g \alpha}{2 \gamma_g}\right)\right]e^{-\lambda T} + 2\eta\cos{2\pi j \over 8} e^{i \psi}.
\end{equation}
Then, the restriction of the characteristic quasi-polynomial to $V_{j,1}$ reads
\begin{equation}\label{eq:quasi-polynomial-restriction}
\mathcal P_{j,1}(\alpha,\lambda,\mathcal O(\alpha)) = 
\begin{cases}
\widetilde{\mathcal P_j}, \quad {\rm if} \;\;\;\;\; j = 0,4,\\
(\widetilde{\mathcal P_j})^2 \quad {\rm if}\;\; j = 1,2,3,
\end{cases}
\end{equation}
so that 
\begin{equation}\label{eq:quasi-polynomial-full-S}
\mathcal P_{\ast}(\alpha,\lambda,\mathcal O(\alpha)) = \prod_{j = 0}^4 \mathcal P_{j,1}(\alpha,\lambda,\mathcal O(\alpha)). 
\end{equation}

\medskip
\noindent
{\bf Condition (A2): existence of centers.}
In order to simplify the notations (cf. \eqref{eq:mm}), put
\begin{equation}\label{eq:x-alpha}
x(\alpha):= \frac{\alpha}{2\gamma_g}-\frac{q_o}{2\gamma_q},\qquad y(\alpha):= \frac{  \eta_qq_o}{2\gamma_q}-\frac{  \eta_g \alpha}{2\gamma_g},
\end{equation}
and
\begin{equation}\label{eq:a-j}
a_j+ib_j:= 2\eta e^{i\psi} \cos \frac{2\pi j}{8}.
\end{equation}
Let us identify the values of $\alpha$ for which $\mathcal O(\alpha)$ is a center, i.e. we are looking for those values of $\alpha$ for which there exists $w>0$ such that 
$\mathcal P_*(\alpha,iw,\mathcal O(\alpha))=0$.
Equivalently (cf.~\eqref{eq:x-alpha}-\eqref{eq:a-j}), 
\[
iw=-\gamma+\gamma \sqrt \kappa \exp(x(\alpha) +i(y(\alpha)-wT)) +a_j+ib_j, \quad j=0,1,2,3,4.
\] 
This complex equation can be reduced to the real equation
\begin{equation}\label{eq:center4}
\tan(y(\alpha)-w(\alpha)T)=\frac {w(\alpha)-b_j}{\gamma - a_j}, \quad j=0,1,2,3,4,
\end{equation}
with
\begin{equation}\label{eq:center3}
w(\alpha):=\gamma\sqrt\kappa e^{x(\alpha)}\sqrt{1-\frac{(\gamma-a_j)^2}{\gamma^2\kappa e^{2x(\alpha)}}}+b_j.
\end{equation}

For $\alpha$ large enough, the right-hand side of \eqref{eq:center4} is close to 
${\gamma\sqrt\kappa \over \gamma - a_j}e^{ \frac{\alpha}{2\gamma_g}-\frac{q_o}{2\gamma_q}}$
(see \eqref{eq:x-alpha} and \eqref{eq:center3}). Combining this with periodicity of the function tangent,  one concludes
that   \eqref{eq:center4} has infinitely many 
solutions $\alpha$ together with the corresponding limit frequencies $w(\alpha)$. 

\begin{proposition}\label{prpr}
	Suppose $\alpha=\alpha_o^j$ is a root of \eqref{eq:center4}, \eqref{eq:center3} for some  $j=0,1,2,3,4$ and
	\begin{equation}\label{eq-transvers}
	\gamma > 2\eta \cos(\psi) \cos{2\pi j \over 8} \qquad {\rm and} \qquad  \omega(\alpha_o^j) > 2\eta \cos(\psi) \sin{2\pi j \over 8}.
	\end{equation}
Then, the following continuous branches of relative equilibria  
bifurcate from the equilibrium $(\alpha_o^j,\mathcal{O}(\alpha_o^j))$ of equation \eqref{eq:f-Dn}:

\smallskip

for $j=0$, a branch with symmetry $(D_8)$;

\smallskip

for $j=1$, two branches with symmetry $(\mathbb Z_8^{t_1})$, four branches with symmetry $(D_2^d)$ and  four branches with symmetry $(\widetilde{D}_2^d)$; 

\smallskip

for $j=2$, two branches with symmetry  $(\mathbb Z_8^{t_2})$, two branches with symmetry $(D_4^d)$ and two branches with symmetry $(\widetilde{D}_4^d)$; 

\smallskip

for $j=3$,  two branches with symmetry $(\mathbb Z_8^{t_3})$, four branches with symmetry $(D_2^d)$ and  four branches with symmetry $(\widetilde{D}_2^d)$; 

\smallskip

for $j=4$, a branch with symmetry  $(D_8^d)$.
\end{proposition}

\begin{proof}
Let us show that the center $\mathcal{O}(\alpha_o^j)$ is isolated (cf.~condition (A2)). 
Put $\lambda(\alpha):=\boldsymbol r(\alpha)+iw(\alpha)$ and rewrite the characteristic equation as follows
(cf. \eqref{eq:mm}-\eqref{eq:a-j}):
\begin{equation}\label{eq:center5}
\begin{cases}
\boldsymbol r(\alpha)=-\gamma+\gamma \sqrt\kappa e^{x(\alpha)-\bold r(\alpha)T}\cos (y(\alpha)-w(\alpha)T)+a_j,\\
w(\alpha)=\gamma \sqrt \kappa e^{x(\alpha)-\boldsymbol r(\alpha)T}\sin(y(\alpha)-w(\alpha)T)+b_j,
\end{cases} 
\end{equation}
where $j=0,1,2,3,4$. 
Assume that for $\alpha=\alpha_o$, the equilibrium $\mathcal O(\alpha_o)$ is a center with the limit frequency $w(\alpha_o)=w_o$
and put 
\begin{equation}\label{eq:notation}
x_o:=x(\alpha_o^j),\quad y_o:=y(\alpha_o^j),\quad x_o':=x'(\alpha_o^j)=\frac 1{2\gamma_g},\quad y_o':=y'(\alpha_o^j)=-\frac{\eta_g}{2\gamma_q}.
\end{equation} 
Differentiating \eqref{eq:center5} with respect to $\alpha$, 
one obtains
\[
\begin{cases}
\boldsymbol r'(\alpha_o^j)&= (\gamma-a_j)(x'_o-\boldsymbol r'(\alpha_o^j)T)-(w_o-b_j)(y'_o-w'(\alpha_o^j)T),\\
w'(\alpha_o^j)&=(w_o-b_j)(x'_o-\boldsymbol r'(\alpha_o^j)T)+ (\gamma-a_j)(y'_o-w'(\alpha_o^j)T),
\end{cases}
\]
which leads to 
\begin{equation}\label{eq:center9}
\boldsymbol r'(\alpha_o^j)=\frac{[(\gamma-a_j)^2x'_o+(w_o-b_j)^2y'_o]T+(\gamma-a_j)x_o'-(w_o-b_j)y_o'}{(1+(\gamma-a_j)T)^2+(w_o-b_j)^2T^2}.
\end{equation}
Formulas \eqref{eq:notation}, \eqref{eq:center9}  show that $\boldsymbol r'(\alpha_o^j) > 0$
provided that relations \eqref{eq-transvers}
are satisfied.
Hence, relations \eqref{eq-transvers} guarantee that the transversality condition for $\lambda(\alpha^j)$ 
is satisfied at $\alpha = \alpha_o^j$, in which case the center $\mathcal O(\alpha_o^j)$ is isolated.  Moreover, 
relations  \eqref{eq-transvers}  imply that condition (iii) from Proposition \ref{prop-bifurcation-relative-equib} is satisfied. Since the other conditions have been verified, the result follows.
\end{proof}

\vs
Recall that $\eta$ stands for the coupling strength.
In particular, conditions \eqref{eq-transvers} are satisfied for all $j$ for any relatively weak coupling.


Table \ref{zero:eq:tab} illustrates Proposition \ref{prpr}.
Assume that 
$\eta = 2$, $\alpha_g = 1$, $\alpha_q = 1$, $\gamma_g = 10^{-2}$, $\gamma_q = 1$, $\gamma=15$, $\kappa = \sqrt{0.2}$, $q_0 = 2$, $E_g=1$, $E_q=0.1$, $T=2.5$. For this set of parameters conditions \eqref{eq-transvers} are fulfilled for all $(\alpha,\omega(\alpha))$ satisfying equations \eqref{eq:center4}, \eqref{eq:center3}.
For $\alpha<0.036$, the equilibrium $\mathcal{O}(\alpha)$ is stable. 
In Table \ref{zero:eq:tab}, we localize Hopf bifurcation points along the horizontal direction, and specify isotypical components $V_{j,1 }$  along the vertical direction. In each cell, we indicate the 
number of unstable roots of the corresponding characteristic polynomial $\mathcal{P}_{j,1}$ defined by \eqref{eq:quasi-polynomial-restriction}.
One can easily see a change of stability as $\alpha$ increases. 
An entry of the table is circled to indicate  a ``jump" in the number
of unstable roots and hence a Hopf bifurcation point. In particular,   Proposition \ref{prpr}  guarantees
Hopf bifurcations of branches of relative equilibria as follows:
\begin{enumerate}[(i)]
\item with symmetry $(D_8)$ for $ 
\alpha\approx 0.03606$;
\item  with symmetries $(\mathbb Z^{t_1}_8), (D_2^d)$ and  $(\widetilde{D}_2^d)$ for $ 
\alpha \approx 0.03607$;
\item with symmetries $(\mathbb Z^{t_2}_8), (D_4^d)$ and  $(\widetilde{D}_4^d)$ for $ 
\alpha \approx 0.0361$;
\item with symmetries $(\mathbb Z^{t_3}_8), (D_2^d)$ and  $(\widetilde{D}_2^d)$ for $ 
\alpha \approx 0.03613$;
\item with symmetry $(D_8)$ for $
\alpha \approx 0.03617$;
\item with symmetries $(D_8^d)$ and $(\mathbb Z^{t_1}_8)$ for $
\alpha \approx 0.0362$,
\end{enumerate}
\label{pageI}

\noindent
to mention a few (see Proposition \ref{prpr} for the number of branches of each type).

\begin{table}[ht!]
{\setlength{\tabcolsep}{0.5em} 
\centering
\caption{Number of unstable eigenvalues in each isotypical component for the equilibrium $\mathcal{O}(\alpha)$\label{zero:eq:tab}}
{\small
\begin{tabular}{|c|c|c|c|c|c|c|c|c|c|c|c|c|c|c|c|c|c|}
\cline{3-9}
\multicolumn{2}{c|}{} &\multicolumn{7}{c|}{Intervals for values of parameter $\alpha \; \cdot 10^2$ } \\ \cline{3-9}
\multicolumn{2}{c|}{} & \scriptsize $[3.6,\; 3.606]$ & \scriptsize $[3.6065, \; 3.607]$ & \scriptsize $[3.6075,\; 3.6095]$ & \scriptsize $[3.61,\; 3.613]$ &  \scriptsize $[3.6135,\; 3.617]$ &  \scriptsize $[3.618,\; 3.62]$ &  \scriptsize $[3.6205,\;3.622]$ \\ \hline 
\multirow{6}{*}{\rot{Isotypical component}}&$V_{0, 1}$ & 0 & \circled{2} & 2 & 2 & 2 & \circled{4} & 4  \\ \cline{2-9}
&$V_{1, 1}$ & 0 & 0 & \circled{4} & 4 & 4 & 4 & \circled{8}  \\ \cline{2-9}
&$V_{2, 1}$ & 0 & 0 & 0 & \circled{4} & 4 & 4 & 4  \\\cline{2-9}
&$V_{3, 1}$ & 0 & 0 & 0 & 0 & \circled{4} & 4 & 4 \\ \cline{2-9}
&$V_{4, 1}$ & 0 & 0 & 0 & 0 & 0 & 0 & \circled{2}   \\ \cline{2-9}
&$ \bigoplus\limits_{j=0}^4 V_{j, 1}$ & 0 & 2 & 6 & 10 & 14 & 16 & 22 \\ \hline
\end{tabular}
}}
\end{table}

\subsection{Bifurcation of relative periodic solutions} 
\subsubsection{Application of Theorem \ref{t1} to the laser system}\label{s6}
In this subsection, we show how Theorem \ref{t1} can be applied to classify symmetries of relative periodic solutions, which bifurcate from branches of relative equilibria of system \eqref{eq:f-Dn} with $n=8$. 
We  restrict the presentation to bifurcations from  relative equilibria that have 3 particular types of symmetry, $(D_8)$, $(\mathbb{Z}^{t_1}_8)$, and $(D_8^d)$. These branches are listed under the items (i), (ii), and (vi), respectively, on page \pageref{pageI}.
Further, an infinite number of Hopf bifurcations of relative periodic solutions occurs along each branch 
of relative equilibria. To be specific, we consider a few successive Hopf bifurcations at the beginning of each branch of our choice. 
In contrast to 
the  application of Proposition \ref{prop-bifurcation-relative-equib} to studying bifurcation of relative equilibria (in which case, all the necessary symbolic computations were explicitly presented),  
we have to resort to numerical computations for verifying conditions (A3), (A5) and (ii), (iii)
of Theorem  \ref{t1}.

Based on the numerical evidence, Theorem \ref{t1} allows us to predict the following bifurcations of branches of relative periodic solutions. 

\vs
Consider the $(D_8)$-symmetric branch of relative equilibria, which is denoted by (i) on page \pageref{pageI}.
The following branches of relative periodic solutions bifurcate from this branch  (we refer to Section \ref{app:1} for the notation):
	\begin{enumerate}[(i)]
		\item with symmetries   $(\boldsymbol{\mathbb Z_8^{t_1}}), (\boldsymbol{D_2^d}), (\boldsymbol{\widetilde{D}_2^d})$  for $
		\alpha \approx 0.0386$;                                                
		\item  with symmetries    $(\boldsymbol{\mathbb Z_8^{t_2}}), (\boldsymbol{D_4^d}), (\boldsymbol{\widetilde{D}_4^d})$   for $
		\alpha \approx  0.0533$;
		\item  with symmetry  $(\boldsymbol {D_8})$ for  $
		\alpha \approx  0.0602$.
	\end{enumerate}
	
	\vs
	Consider the $(\mathbb Z_n^{t_1})$-symmetric branch of relative equilibria, which is denoted by (ii) on page \pageref{pageI}.
	The following branches of relative periodic solutions bifurcate from this branch:
	\begin{enumerate}[(i)]
		\item with symmetries $(\boldsymbol{\bz_8^{t_1}})$ and $(\boldsymbol{\bz_8^{t_2}})$ for $
		\alpha \approx 0.0366$;
		\item with symmetry $(\boldsymbol{\bz_8^{t_3}})$ for  $
		\alpha \approx 0.0399$;
		\item with symmetries $(\boldsymbol{\bz_8^{t_1}})$ and $(\boldsymbol{\bz_8^{t_3}})$   for  $ 
		\alpha\approx    0.0416 $;
		\item with symmetry $(\boldsymbol{\bz_8^c})$   for  $ 
		\alpha \approx    0.064 $;
		\item with symmetry  $(\boldsymbol{\bz_8^{t_2}})$ for  $ 
		\alpha \approx 0.0641 $;
		\item with symmetry    $(\boldsymbol{\bz_8})$ for  $ 
		\alpha \approx   0.0788.$
	\end{enumerate}
	
	\vs
Consider the $(D_8^d)$-symmetric branch of relative equilibria, which is denoted by (vi) on page \pageref{pageI}.
The following branches of relative periodic solutions bifurcate from this branch:
	\begin{enumerate}[(i)]
		\item with symmetries    $(\boldsymbol{\mathbb Z_8^{t_2}}), (\boldsymbol{D_4^d}), (\boldsymbol{\widetilde{D}_4^d})$ for $
		\alpha \approx 0.0384  $;
		\item with symmetry       $(\boldsymbol {D_8})$      for  $
		\alpha \approx   0.0405  $;
		\item with symmetries      $(\boldsymbol{\mathbb Z_8^{t_3}}), (\boldsymbol{D_2^d}), (\boldsymbol{\widetilde{D}_2^d})$   for  $
		\alpha \approx 0.0539    $;
		\item with symmetry       $(\boldsymbol{D_8^d})$    for  $ 
		\alpha \approx 0.066  $;
		\item with symmetry            $(\boldsymbol{D_8^d})$     for  $
		\alpha \approx   0.0731  $;
		\item with symmetries      $(\boldsymbol{\mathbb Z_8^{t_1}}), (\boldsymbol{D_2^d}), (\boldsymbol{\widetilde{D}_2^d})$   for  $  
		\alpha \approx    0.0757 $.
	\end{enumerate}
	
	Further bifurcations along these and other branches of relative equilibria can be classified in a similar manner.
	 Note that branches of relative periodic solutions with symmetries $(\boldsymbol{\mathbb{Z}_8^{t_j}})$, $(\boldsymbol{D_4^d})$, $(\boldsymbol{\widetilde{D}_4^d})$
	come in pairs, while the branches with symmetries 
	$(\boldsymbol{D_2^d})$, $(\boldsymbol{\widetilde{D}_2^d})$ appear in quadruples.
	
	In the rest of the paper, we show how the above bifurcations can be deduced from Theorem \ref{t1}. 
	Given a relative equilibrium with symmetry group $\mathcal H$, the verification of assumptions of Theorem  \ref{t1} splits into the following steps: (a) finding the isotypical decomposition of $\mathcal H \times S^1$-representation \eqref{decomp'} and providing a list of maximal orbit types in each component (see Subsections \ref{s1} and \ref{subsub-isotypical}); (b) obtaining characteristic quasi-polynomials associated with each isotypical component (see Subsection \ref{s3}); and, (c) analyzing roots of the quasi-polynomials and verifying conditions (A3), (A5), (i) and (ii) of Theorem \ref{t1}	(see Subsections \ref{s5}). The last step relies on numerical computations.  Condition (A3) is reduced to an explicit inequality in Subsection \ref{s4}.      

\subsubsection{Symmetries of relative equilibria}\label{s1}
To begin with, below we will describe {some}  of the relative equilibria identified in the previous subsection 
{more explicitly}.  

Observe first that the group $D_n$ described in Subsection \ref{subsect-D-n-configuration} can be identified (for convenience) with
$D_n = \{1,\xi, ..., \xi^{n-1}, \kappa,\xi\kappa,...,\xi^{n-1}\kappa\}$, where 
\begin{equation}\label{D-n-identification}
\xi: = e^{i {2\pi \over n}} = 
\begin{bmatrix}
\cos({2\pi \over n}) & -\sin({2\pi \over n})\\ \sin({2\pi \over n}) & \cos({2\pi \over n})
\end{bmatrix} \quad {\rm and} \quad 
\kappa = \begin{bmatrix} 1 & 0 \\ 0 & -1\end{bmatrix}.
\end{equation}
Let ${\bf S}(\overline{x})$ be a relative equilibrium of system \eqref{eq:f-Dn} (cf. Remark \ref{rem:symmetries-relative-equilib}). 
Fix an integer $l$ satisfying  $0 \leq l < n$, put $\zeta:= \xi^l$ and assume that 
\begin{equation}\label{eq:rel-eq-l}
\overline x:=(\overline x^o,\zeta\overline x^o,\zeta^2\overline x^o,\dots,\zeta^{n-1}\overline x^o)^\top, \quad  \overline x^o = (g,q,a) \in \mathbb R\oplus \mathbb R \oplus \mathbb C \simeq \mathscr  V,\ a \not=0.
\end{equation}
One can easily verify that in this case, under the $\mathcal G := D_n \times S^1$-action, the isotropy of $\overline x$ is completely determined by the relations
 \[
 (\zeta^k,z)  \in \mathcal G_{\overline{x}}   \;\;\; \Leftrightarrow\;\;\; z\zeta^{-k}=1\;\;\; \Leftrightarrow\;\;\; z=\xi^{-lk},
 \]
 where $0\le k\le n-1$, and 
 \[
 (\kappa,z) \in \mathcal G_{\overline{x}} \;\;\; \Leftrightarrow\;\;\; l=0\;\text{ and }\; z=1.
 \]
By direct verification, $\overline{x}$ is of the form \eqref{eq:rel-eq-l} if and only if
\begin{equation}\label{eq:possible-symm-rel-eq}
(\mathcal G_{\overline x}):=
\begin{cases}
(D_n\times\{1\}) \simeq (D_n) \;&\text{ if } \; l=0,\\
D_n^d & \text{ if } \; l=\frac n2,\\
\bz_n^{t_l}:=\{ (\xi^k,\xi^{kl})\in D_n\times S^1: k=0,1,\dots,n-1\} &\text{otherwise.} 
\end{cases}
\end{equation}

\begin{remark}\label{eq:other-symmetric-rel-equilib} 
\rm In what follows, as in Subsection \ref{subsec:bif-rel-equil}, we will restrict ourselves to the case $n = 8$. As it was established in Subsection 
\ref{subsec:bif-rel-equil}, twisted subgroups listed in \eqref{eq:possible-symm-rel-eq} do not exhaust possible symmetries of relative equilibria of
system \eqref{eq:f-Dn}. For example, one can easily check that if 
\[\overline{x} = (\overline{x}^1,\overline{x}^2,\overline{x}^3,\overline{x}^4,-\overline{x}^1,-\overline{x}^2,-\overline{x}^3,-\overline{x}^4), \quad \quad\quad \overline{x}^i \in \mathcal V,\] 
then $(\mathcal G_{\overline{x}}) = (D_2^d)$, and if
\[\overline{x} =  (\overline{x}^1,-\overline{x}^1,\overline{x}^2,-\overline{x}^2,\overline{x}^3,-\overline{x}^3,\overline{x}^4,-\overline{x}^4), \quad \quad\quad \overline{x}^i \in \mathcal V, \] 
then 
$(\mathcal G_{\overline{x}}) = (\widetilde{D}_2^d)$. However, in order to keep our paper reasonably simple 
and of appropriate size, we omit these cases. 
\end{remark}
 

\subsubsection{$\mathcal G_{\overline{x}}$-isotypical decomposition of $V^c$ and maximal twisted orbit types}\label{subsub-isotypical}

\paragraph{(a) Identification.} In this subsubsection, we describe the $\mathcal H$-isotypical decomposition of the space $V^c$, where $\mathcal H=D_n\times \{1\}$, $D^d_n$ and $\bz_n^{t_l}$ with $0< l <\frac n2$ (cf. \eqref{eq:possible-symm-rel-eq}). We will assume that $n>2$ is an even integer and put $r:=\frac n2$. 
Notice that $D_n\times \{1\}$ and $D_n^d$ can be identified with $D_n$ while $\bz^{t_l}_n$ can be identified with $\bz_n$. 
\vs

Complex irreducible $\bz_n$-representations $\mathcal U'_j$ can be easily described: (a) the trivial representation $\mathcal U'_0=\bc$,  (b) 
$\mathcal U'_r=\bc$ with the natural antipodal action of $\bz_2:=\bz_n/\bz_r$,  and (c) $\mathcal U'_{\pm j}=\bc$, where the $\bz_n$-action is given by 
\[
\xi z=\xi^{\pm j}\cdot z, \quad z\in \bc.
\]
In the case of the group $D_n$, we have the following irreducible $D_n$-representations:
(a) the trivial representation $\mathcal U_0=\bc$, (b) the representation $\mathcal U_r=\bc$ with the natural antipodal action of  $\bz_2:=D_n/D_r$, 
and (c) the representations 
$\mathcal U_{ j}=\bc\oplus \bc$ for $0<j<r$ with the  $D_n$-action given by
\[
\xi(z_1,z_2)=(\xi^{j}\cdot z_1,\xi^{- j}\cdot z_2), \quad \kappa(z_1,z_2)=(z_2,z_1) \quad (z_1,z_2\in \bc).
\]
Notice that $\bz_n\subset D_n$, therefore,  for $0<j<r$, we have the decomposition
\[
\mathcal U_j=\mathcal U'_j\oplus \mathcal U'_{-j}.
\] 
We do not consider other (one-dimensional) irreducible $D_n$-representations since they are irrelevant for the decomposition of the substitutional 
$D_n$-representation we are dealing with in what follows.
\vs
If $\mathcal H \simeq \mathbb Z_n$, then the complex $\mathcal H$-representation $V^c$ admits the following $\bz_n$-isotypical decomposition
\begin{align}\label{eq:Zn-decoposition}
V^c&=U_0\oplus U_1^+\oplus U_1^-\oplus \dots \oplus U_{r-1}^+\oplus U_{r-1}^-\oplus U_r,
\end{align}
where the components $U_j^{\pm }$ (resp. $U_0$ and $U_r$) are modeled on the complex irreducible $\bz_n$-representation 
$\mathcal U'_{\pm j}$ (resp. $\mathcal U'_0$ and $\mathcal U'_r$). Furthermore,
if $\mathcal H \simeq D_n$, then 
\begin{equation}\label{eq:H-decomposition}
V^c=U_0\oplus U_1\oplus\dots\oplus U_{r-1}\oplus U_r,
\end{equation}
where $U_j=U_j^+\oplus U_j^-$ for $0<j<r$ and the isotypical component $U_j$ is modeled on the irreducible $D_n$-representation $\mathcal U_j$. 
Also, $U_0$ and $U_r$ are modeled on $\mathcal U_0$ and $\mathcal U_r$, respectively. 

\begin{remark}\label{rem:mathcalVc-decomposition} \rm (i) The complexification $\mathscr V^c$ of the space $\mathscr V:= \br^2\oplus \bc=\br^2\oplus (\br\oplus\br)$ can be represented as
\begin{equation}\label{eq:complexif-mathcalV}
\mathscr V^c = \bc^2 \oplus \Big(\bc\oplus \bc \Big)=\bc^4,
\end{equation} 
thus $V^c=(\mathscr V^c)^n = (\bc^4)^n$ for which decomposition \eqref{eq:Zn-decoposition} takes place.

(ii) Any complex $\mathcal H$-equivariant linear operator $A:V^c
 \to V^c$ is also $\bz_n$-equivariant, thus it preserves isotypical decomposition \eqref{eq:Zn-decoposition}. 
 
(iii) Clearly, the space $V^c$ admits a natural $S^1$-action induced by the complex multiplication. Put $\mathcal K:=\mathcal H\times S^1$.
Then (cf. \eqref{decomp'}), the $S^1$-action converts the (complex) $\mathcal H$-isotypical decomposition \eqref{eq:H-decomposition} 
into a (real) $\mathcal K$-isotypical decomposition
\begin{equation}\label{eq:H-decomposition-R}
V^c=U_{0,1}\oplus U_{1,1}\oplus\dots\oplus U_{r-1,1}\oplus U_{r,1}.
\end{equation}

(iv) By inspection, for $n = 8$ (our case study),  if $(H)$ is a {maximal} twisted orbit type 
in an {isotypical component} of $V^c$, then $(H)$ is a maximal twisted  orbit type in $V^c$ itself.

\end{remark}

\paragraph{(b) $\mathcal H := D_n\times \{1\}$-isotypical decomposition of $V^c$.} 
One can explicitly  describe the $\mathcal H$-isotypical components of \eqref{eq:H-decomposition} as follows:
\[
U_0 =\{(z,z,\dots,z)^\top : z\in \bc^4   \},
\]
\[U_j = U_j^+\oplus U^-_j, \quad U^\pm_j  := \{(z,\xi^{\pm j} z,\dots,\xi^{\pm j(n-1)} z)^\top: z \in \mathbb C^4 \} \quad (0 < j < r),
\]
\[
U_r =\{(z,-z,z,-z,\dots,z,-z)^\top: z\in \bc^4\}.
\]
Further, one can easily verify that the coupling matrix $\mathscr C:V^c\to V^c$ given by \eqref{eq:C} preserves the $\mathcal H$-isotypical components. Put 
\begin{equation}\label{eq:rest-coupling}
\mathscr C^\pm_j := 
\mathscr C|_{U^\pm_j},\;\;  \mathscr C_0:= \mathscr C|_{U_0},\;\;  \mathscr C_r:= \mathscr C|_{U_r} \quad (0 < j < r).
\end{equation} 
Then,
\[
\mathscr C^\pm_j=\left[\begin{array}{cccc} 0&0&0&0\\0&0&0&0\\
0&0&2a_j\cos \psi&-2a_j\sin\psi\\
0&0&2a_j\sin \psi&2a_j \cos\psi\end{array}   \right], \quad a_j = {\rm Re}(\xi^{j})=\cos\frac{2\pi j}{n} \;\;\;\; (0 < j < r),
\]
\[
\mathscr C_0=\left[\begin{array}{cccc} 0&0&0&0\\0&0&0&0\\
0&0&2\cos \psi&-2\sin\psi\\
0&0&2\sin \psi&2 \cos\psi\end{array}   \right], \quad
\mathscr C_r=\left[\begin{array}{cccc} 0&0&0&0\\0&0&0&0\\
0&0&-2\cos \psi&2\sin\psi\\
0&0&-2\sin \psi&-2 \cos\psi\end{array}   \right].
\]
Finally, for $n = 8$, the list of maximal twisted types in the isotypical components of the $\mathcal H:=D_8 \times \{1\}$-representation  $V^c$
is as follows (see Appendix, 
Subsection \ref{notation:1}, for the exact definition of the related twisted subgroups): 

\smallskip
(i)  for $U_{0,1}$: \quad $(\boldsymbol {D_8})$;

\smallskip

(ii) for $U_{1,1}$: \quad $(\boldsymbol{\mathbb Z_8^{t_1}}), (\boldsymbol{D_2^d}), (\boldsymbol{\widetilde{D}_2^d})$; 

\smallskip

(iii) for $U_{2,1}$: \quad $(\boldsymbol{\mathbb Z_8^{t_2}}), (\boldsymbol{D_4^d}), (\boldsymbol{\widetilde{D}_4^d})$; 

\smallskip

(iv)  for $U_{3,1}$: \quad $(\boldsymbol{\mathbb Z_8^{t_3}}), (\boldsymbol{D_2^d}), (\boldsymbol{\widetilde{D}_2^d})$; 

\smallskip

(v)  for $U_{4,1}$: \quad $(\boldsymbol{D_8^d})$.

\vs 
\paragraph{(c) $\mathcal H:=\bz_n^{t_l}$-isotypical decomposition of $V^c$.} For this group $\mathcal H$, the $\mathcal H$-isotypical components of 
\eqref{eq:H-decomposition} can be described as follows (cf. \eqref{eq:complexif-mathcalV}):
\[
U_0 = \mathscr U_0 \oplus \mathscr W_0,
\]
where
\[
\mathscr U_0 := \{(z,z,...,z)^\top \, : \, z \in \mathbb C^2\}
\]
and
\begin{align*}
\mathscr W_0&:= \left\{\left(\left[\begin{array}{c}z_1\\z_2\end{array}\right],\left[\begin{array}{c}\xi^l z_1\\ \xi^{-l} z_2\end{array}\right],...,\left[\begin{array}{c}\xi^{(n-1)l} z_1\\
\xi^{-(n-1)l} z_2\end{array}\right]\right)^\top\; : \; \left[\begin{array}{c}z_1\\ z_2\end{array}\right] \in \mathbb C \oplus \mathbb C\right\};
\end{align*}
\[
U^{\pm }_j:=\mathscr U^{\pm}_j\oplus \mathscr W^\pm_j \quad (0 < j < r),
\]
where 
\[
\mathscr U^\pm _j=\left\{(z,\xi^{\pm j} z,\xi^{\pm 2j}z, \dots,\xi^{\pm (n-1)j} z)^\top: z\in \bc^2   \right\}
\]
and
\[
\mathscr W^\pm_j := \left\{
\left(\left[\begin{array}{c} z_1\\ z_2\end{array}\right],\left[\begin{array}{c} \xi^{  \mp j+l} z_1\\ \xi^{ \mp j-l} z_2\end{array}\right],\dots,\left[\begin{array}{c}  \xi^{(n-1)(  \mp j+l)} z_1\\  \xi^{(n-1)(  \mp j-l)}z_2\end{array}\right]\right)^\top \; : \; \left[\begin{array}{c}z_1\\ z_2\end{array}\right] \in \mathbb C \oplus \mathbb C\right\};
\]
\medskip
\[U_r := \mathscr U_{r}\oplus \mathscr W_{r},\]
where
\[
 \mathscr U_{r}:= \{(z,-z,z,-z,...,z,-z)^\top \; : \; z \in \mathbb C^2\}
\]
and
\[
\mathscr W_{r}:= \left\{\left(\left[\begin{array}{c}z_1\\z_2\end{array}\right],\left[\begin{array}{c} -\xi^l z_1\\ -\xi^{-l} z_2\end{array}\right],\left[\begin{array}{c} \xi^{2l} z_1\\ \xi^{2l} z_2\end{array}\right], \dots, 
\left[\begin{array}{c}-\xi^{l(n-1)}z_1\\ -\xi^{-l(n-1)}z_2\end{array}\right]\right)^\top : 
\left[\begin{array}{c}z_1\\z_2\end{array}\right]\in \bc\oplus \bc \right\}.
\]
Under the same notations as in \eqref{eq:rest-coupling}, one has
\[
\mathscr C^{\pm}_j :=\left[\begin{array}{cccc} 0&0&0&0\\0&0&0&0\\
0&0&2a_{\pm j}\cos \psi&-2a_{\pm j}\sin\psi\\
0&0&2a_{\pm j}\sin \psi&2a_{\pm j} \cos\psi\end{array}   \right], \quad a_{\pm j}=\cos\frac{2\pi (\pm j-1)l}{n} \;\;\;\; (0 < j < r),
\]
\[
\mathscr C_0 :=\left[\begin{array}{cccc} 0&0&0&0\\0&0&0&0\\
0&0&2a_0\cos \psi&-2a_0\sin\psi\\
0&0&2a_0\sin \psi&2a_0 \cos\psi\end{array}   \right], \quad 
 \mathscr C_r :=\left[\begin{array}{cccc} 0&0&0&0\\0&0&0&0\\
0&0&2a_r\cos \psi&-2a_r\sin\psi\\
0&0&2a_r\sin \psi&2a_r \cos\psi\end{array}   \right], 
\]
where $a_0:= \cos{2\pi l \over n}$ and $a_r:= -\cos{2\pi l \over n}$.

For $n = 8$, one obtains the following list of maximal twisted types in the isotypical components of the $\mathcal H:=\bz_8^{t_l}\times \{1\}$-representation  
$V^c$, $l=1,2,3$ (see Appendix, Subsection \ref{notation:2},
for the definition of the related twisted subgroups): 

\smallskip
(i)  for $U_{0,1}$: \quad $(\boldsymbol{\bz_8})$;

\smallskip

(ii) for $U_{1,1}$: \quad $(\boldsymbol{\bz_8^{t_1}})$;

\smallskip

(iii) for $U_{2,1}$: \quad $(\boldsymbol{\bz_8^{t_2}})$;

\smallskip

(iv)  for $U_{3,1}$: \quad $(\boldsymbol{\bz_8^{t_3}})$;

\smallskip

(v)  for $U_{4,1}$: \quad $(\boldsymbol{\bz_8^c})$.

\vs

\paragraph{(d) $\mathcal H:=D_n^d$-isotypical decomposition of $V^c$.}
In this case, one can explicitly  describe the $\mathcal H$-isotypical components of \eqref{eq:H-decomposition} as follows:
\[
U_0 = \mathscr U_0 \oplus \mathscr W_0,
\]
where
\[
\mathscr U_0 := \{(z,z,...,z)^\top \, : \, z \in \mathbb C^2\}
\]
and
\[ 
\mathscr W_0:= \left\{\left(\left[\begin{array}{c}z_1\\z_2\end{array}\right],\left[\begin{array}{c}-z_1\\ -z_2\end{array}\right],\left[\begin{array}{c}z_1\\z_2\end{array}\right],..., \left[\begin{array}{c}- z_1\\ -z_2\end{array}\right]\right)^\top\; : \; \left[\begin{array}{c}z_1\\z_2\end{array}\right] \in \mathbb C \oplus \mathbb C\right\};
\]
\medskip
\[
U^\pm_j:=\mathscr U^\pm_{ j}\oplus \mathscr W^\pm _{j} \quad (0 < j < r),
\]
where
\[
\mathscr U^\pm _{j}=\left\{(z,\xi^{\pm j} z,\xi^{\pm 2j}z, \dots,\xi^{\pm (n-1)j} z)^\top: z\in \bc^2   \right\}
\]
and
\[
 \mathscr W^\pm_j := \left\{\left(\left[\begin{array}{c} z_1\\z_2\end{array}\right],\left[\begin{array}{c} -\xi^{ \mp j} z_1\\- \xi^{\mp j} z_2\end{array}\right],\dots, \left[\begin{array}{c}(-\xi^{ \mp j})^{n-1} z_1\\(-\xi^{\mp j})^{n-1}z_2\end{array}\right]\right)^\top: 
 \left[\begin{array}{c}z_1\\z_2\end{array}\right] \in \mathbb C \oplus \mathbb C\right\};
\]
\medskip
\begin{equation}
U_r:=\mathscr U_{r}\oplus \mathscr W_{r},
\end{equation}
where
\[
 \mathscr U_{r}:= \{(z,-z,z,-z,...,z,-z)^\top \; : \; z \in \mathbb C^2\}
\]
and
\[
\mathscr W_{r}:=  \left\{\left(\left[\begin{array}{c} z_1\\z_2\end{array}\right],\left[\begin{array}{c}z_1\\z_2\end{array}\right],...,\left[\begin{array}{c}z_1\\z_2\end{array}\right]\right)^\top: 
 \left[\begin{array}{c}z_1\\z_2\end{array}\right] \in \mathbb C \oplus \mathbb C\right\}.  
\]
Also,
\[
\mathscr C^{\pm}_j :=\left[\begin{array}{cccc} 0&0&0&0\\0&0&0&0\\
0&0&2a_{\pm j}\cos \psi&-2a_{\pm j}\sin\psi\\
0&0&2a_{\pm j}\sin \psi&2a_{\pm j} \cos\psi\end{array}   \right], \quad a_{\pm j} = -\cos\frac{2\pi j}{n} \;\;\;\; (0 < j < r),
\]
\[
\mathscr C_0 :=\left[\begin{array}{cccc} 0&0&0&0\\0&0&0&0\\
0&0&-2\cos \psi& 2\sin\psi\\
0&0&- 2\sin \psi & -2 \cos\psi\end{array}   \right], \quad 
 \mathscr C_r :=\left[\begin{array}{cccc} 0&0&0&0\\0&0&0&0\\
0&0& 2 \cos \psi &-2\sin\psi\\
0&0&2\sin \psi & 2 \cos\psi\end{array}   \right]. 
\]
Hence, for $n = 8$, the list of maximal twisted types in the isotypical components of the $\mathcal H:=D_n^d$-representation  $V^c$
is (see Appendix, Subsection \ref{notation:3}, for the definition of the twisted subgroups): 

\smallskip
(i)  for $U_{0,1}$: \quad $(\boldsymbol {D_8})$;

\smallskip

(ii) for $U_{1,1}$: \quad $(\boldsymbol{\mathbb Z_8^{t_1}}), (\boldsymbol{D_2^d}), (\boldsymbol{\widetilde{D}_2^d})$; 

\smallskip

(iii) for $U_{2,1}$: \quad $(\boldsymbol{\mathbb Z_8^{t_2}}), (\boldsymbol{D_4^d}), (\boldsymbol{\widetilde{D}_4^d})$; 

\smallskip

(iv)  for $U_{3,1}$: \quad $(\boldsymbol{\mathbb Z_8^{t_3}}), (\boldsymbol{D_2^d}), (\boldsymbol{\widetilde{D}_2^d})$; 

\smallskip

(v)  for $U_{4,1}$: \quad $(\boldsymbol{D_8^d})$.
\vs

\subsubsection{Linearization on a relative equilibrium and characteristic quasi-polynomials}
\label{s3}

For any $\overline x^o = (g,q,a) \in \mathbb R\oplus \mathbb R \oplus \mathbb C \simeq \mathscr  V$, one has (cf. \eqref{eq:xi1}--\eqref{d} and 
\eqref{eq:ML-a}):
\begin{equation}\label{eq:map-in-mathcal-V}
\widetilde{\mathfrak f}(\alpha,i\omega,\overline{x}^o) =  
\begin{bmatrix}
\alpha - \gamma_g  g -\frac 1{E_g} e^{-q}(e^{g}-1)|a|^2 \\
q_0-\gamma_q q-\frac 1{E_q}\left( 1-e^{-q}  \right)|a|^2\\
-\gamma a + \gamma\sqrt \kappa \exp\left[\frac{(1-i  \eta_g)g - (1 - i  \eta_q)q} {2}\right] a e^{-i \omega T}
\end{bmatrix}.
\end{equation}
Take $\lambda \in \mathbb C$. Combining \eqref{eq:map-in-mathcal-V} with \eqref{eq:B-alpha}, \eqref{eq:R-alpha}  and \eqref{eq:R(0)} allows us to define a ``linearization operator"
$\mathcal R_{\alpha}^{\mathscr  V}(\lambda): \mathscr  V^c \to \mathscr  V^c$ by 
\begin{equation}\label{eq:linearization-on-rel-equib-mathcal-V}
\mathcal R_{\alpha}^{\mathscr  V}(\lambda): = 
\begin{bmatrix}
- \gamma_g  -\frac 1{E_g} e^{-q}e^{g}|a|^2  & {1 \over E_g} e^{-q}(e^g - 1) |a|^2 & -{2 \over E_g} e^{-q}(e^g -1)a \\
0 &  -\gamma_q - {1 \over E_q}e^{-q} |a|^2 &  -{2 \over E_q}(1 - e^{-q})a\\
B_{31}(\lambda) & B_{32}(\lambda) & B_{33}(\lambda)
\end{bmatrix},
\end{equation}
where 

$B_{31}(\lambda) = {\gamma\sqrt \kappa (1-i  \eta_g) \over 2}\exp\left[\frac{(1-i  \eta_g)g - (1-i  \eta_q)q}2\right]a e^{-i\omega T - \lambda T}$;

$B_{32}(\lambda) = -{\gamma\sqrt \kappa (1-i  \eta_q) \over 2}\exp\left[\frac{(1-i  \eta_g)g - (1-i  \eta_q)q}2\right]a e^{-i\omega T - \lambda T}$;

$B_{33}(\lambda) = -\gamma + \gamma\sqrt \kappa \exp\left[\frac{(1-i  \eta_g)g - (1-i  \eta_q)q}2\right]a e^{-i\omega T - \lambda T}$.

\medskip\noindent
For any $\overline x = (\overline {x}^1,...,\overline {x}^n) \in V$, put 
\begin{equation}\label{eq:wild-f-o}
 \widetilde{f}_o(\alpha,i\omega,\overline{x}) := 
(\widetilde{\mathfrak f}(\alpha,i\omega,\overline{x}^1),...,\widetilde{\mathfrak f}(\alpha,i\omega,\overline{x}^n))^\top.
\end{equation}
For a given $\alpha$,  ${\bf S}(\overline x(\alpha))$ is a relative equilibrium for system \eqref{eq:f-Dn} corresponding to the frequency $\omega(\alpha)$ if and only if 
\begin{equation}\label{eq:rel-equilib-syst}
\Phi(\alpha,\omega(\alpha),\overline{x}(\alpha)) := \widetilde{f}_o(\alpha,i\omega(\alpha),\overline{x}(\alpha)) + \eta \mathscr C \overline{x}(\alpha) -  
\omega(\alpha)J \overline{x}(\alpha) = 0
\end{equation}
(cf. \eqref{eq:bif-HB}).
Assume that ${\bf S}(\overline{x}(\alpha))$ is a relative equilibrium with $\mathcal H := \mathcal G_{\overline{x}(\alpha)}$ of the form 
\eqref{eq:possible-symm-rel-eq} (cf. \eqref{eq:rel-eq-l} and condition (A4)). Take $ \mathcal R_{\alpha}(\lambda)$ determined by \eqref{eq:rel-equilib-syst}
and  \eqref{eq:B-alpha}--\eqref{eq:R-alpha} and  consider decompositions \eqref{eq:Zn-decoposition}-\eqref{eq:H-decomposition}.  
Then (cf. \eqref{eq:linearization-on-rel-equib-mathcal-V} and \eqref{eq:rel-equilib-syst}), one has:
 \begin{equation}\label{eq:R-alpha-restrict}
 \mathcal R_{\alpha}(\lambda)|_{\mathfrak U} = 
 \begin{cases}
 \mathcal R_{\alpha}^{\mathscr  V}(\lambda) +  \eta \mathscr C_0 \;\,\quad {\rm if}\;\; \mathfrak U = U_0;\\
 \mathcal R_{\alpha}^{\mathscr  V}(\lambda) + \eta \mathscr C^{\pm}_j \quad {\rm if}\;\; \mathfrak U = U^\pm_j \;\;\; (0 < j < r);\\
 \mathcal R_{\alpha}^{\mathscr  V}(\lambda) +  \eta \mathscr C_r \;\,\quad {\rm if}\;\; \mathfrak U = U_r.
 \end{cases}
 \end{equation}
We refer to Subsection \ref{subsub-isotypical}, where explicit formulas for  $\mathscr C_0$, $\mathscr C^{\pm}_j$ and  $\mathscr C_r$ are given
according to three possible values of $\mathcal H$. Combining \eqref{eq:R-alpha-restrict} with \eqref{charact'} and \eqref{barpi}, one can define the
characteristic quasi-polynomials $\overline{\mathcal P}_j(\alpha,\lambda)$, $j = 0,\pm 1,...,\pm (r-1), r$ and study Hopf bifurcation
of relative periodic solutions for different values of $\mathcal H = D_n \times \{1\}, D_n^d, \mathbb Z_n^{t_l}$.

\subsubsection{Condition (A3)}\label{s4} Suppose that equation \eqref{eq:f-Dn} with $n=8$
has a relative equilibrium $\bold S(\overline x)$, $\overline x = (g,q,a)$,  for some $\overline{\alpha}$
and $\overline{w}$. Without loss of generality, assume that $a \in \mathbb C$ is {\it real}.
Take decomposition \eqref{eq-D_n-decomp} and let us describe the restriction of matrix \eqref{xxx} to $\mathbb R \times W_j$. For any $j = 0,...,4$, define the operator
$\mathcal B_j = \mathcal B_j(\overline \alpha,\overline w,\overline x) : \mathcal V \to \mathcal V$ by 
\begin{equation}\label{eq:mathcalB-j}
\mathcal B_j :=  \mathcal R_{\overline{\alpha}}^{\mathscr  V}(0) +   \eta (\xi^j + \xi^{-j}) C  - 
\overline{w} J^{\mathcal V}.
 \end{equation}
Here $ \mathcal R_{\overline{\alpha}}^{\mathscr  V}(0)$ is considered as a {\it real} linear operator in $\mathcal V \simeq \mathbb R^2 \oplus \mathbb C$ 
(cf. \eqref{eq:linearization-on-rel-equib-mathcal-V}) and  $J^{\mathcal V} : \mathcal V \to \mathcal V$ is given by $J^{\mathcal V}(\tilde{g},\tilde{q},\tilde{a})^{\top} = (0,0,i \tilde{a})^{\top}$; see also  \eqref{eq:matrix-C}  and
\eqref{eq:matrix-ksi}. 
Define a vector $\mathcal B = \mathcal B(\overline \alpha,\overline w,\overline x) \in \mathcal V \simeq  \mathbb R^2 \oplus \mathbb C$ by
\begin{equation}\label{eq:mathcalB}
\mathcal B := \Big\{0,0, -iT  \gamma\sqrt \kappa \exp\left[\frac{(1-i  \eta_g)g - (1 - i  \eta_q)q} {2}\right] a e^{-i \omega T} - i a \Big\}^{\top}.
\end{equation} 
Then (see \eqref{xxx}, \eqref{eq:map-in-mathcal-V}, \eqref{eq:mathcalB-j} and \eqref{eq:mathcalB}), 
\begin{equation}\label{eq:restrict-DPhi}
\Big[ D_w \Phi(\overline \alpha,\overline w,\overline x)\;|\;  D_x \Phi(\overline \alpha,\overline w,\overline x)\Big]_{\mathbb R \times W_j} = 
\begin{cases}
[\mathcal  B    \;|\;   \mathcal B_j  
]\ \quad\quad\quad\quad \quad\quad\quad  {\rm if} \; j = 0,4\\
 \left[\begin{array}{cc}
 [\mathcal  B    \;|\;   \mathcal B_j  
 ] & 0\\
 0 & 
 [\mathcal  B    \;|\;   \mathcal B_j  
 ]\\
 \end{array}  \right] \quad {\rm if}\; j = 1,2,3. 
\end{cases}
\end{equation}
Put $\mathfrak B:=  [\mathcal  B    \;|\;   \mathcal B_j ]$. It follows from  \eqref{eq:restrict-DPhi} that condition (A3) is satisfied if 
${\rm rank}(\mathfrak B) = 4$. Note that $(0,0,i)^{\top}\in \mathbb R^2 \oplus \mathbb C$ is an eigenvector of 
$\mathcal B_j$ corresponding to the zero eigenvalue. Denote by $\mathcal {E}$ the direct sum of generalized eigenspaces corresponding to non-zero eigenvalues of $\mathcal B_j$. Clearly, ${\rm rank}(\mathfrak B) = 4$ if 

(a) ${\rm rank}(\mathcal B_j) = 3$ (i.e., zero is a simple eigenvalue of $\mathcal B_j$),  and 

(b) $\mathfrak B e \not\in \mathcal {E}$, where $e := (1,0,0,0,0) \in \mathbb R^5 \simeq \mathbb R \oplus \mathbb R^2 \oplus \mathbb C \simeq \mathbb R \oplus \mathcal V$. 

\begin{remark}\label{rem:A3} {\rm Condition (a) can be effectively expressed in terms of the derivative of the characteristic polynomial
associated with $\mathcal B_j$. 
Condition (b) is satisfied if 

\medskip
(b)$^{\prime}$  ${\rm Im} \Big[ -iT  \gamma\sqrt \kappa \exp\left[\frac{(1-i  \eta_g)g - (1 - i  \eta_q)q} {2}\right] a e^{-i \omega T}\Big] - a \not=0$

\smallskip
\noindent
(recall, $a \in \mathbb R$).}
\end{remark}


\subsubsection{Isotypical crossing}\label{s5}

In order apply Theorem  \ref{t1} to classify symmetries of relative periodic solutions bifurcating from relative equilibria $\bold{S}(\overline{x})$ with $(\mathcal G_{\overline x})$ given by \eqref{eq:possible-symm-rel-eq},  it remains to analyze the isotypical crossing of the roots of characteristic quasi-polynomials $\overline{\mathcal P}_j(\alpha,\lambda)$, $j = 0,\pm 1,...,\pm 3, 4$ (cf. \eqref{eq:R-alpha-restrict}, \eqref{charact'} and \eqref{barpi}), as 
$\alpha$ crosses some critical value $\alpha_o$. Numerical results illustrating isotypical crossing of characteristic roots through 
the imaginary axis are described in Table 2 for   $(\mathcal G_{\overline x}) = D_n \times \{1\}$, in Tables  3, 4, 5 for $(\mathcal G_{\overline x})=\mathbb Z_8^{t_1}, \mathbb Z_8^{t_2}, \mathbb Z_8^{t_3}$, respectively, and  in  Table 6 for $(\mathcal G_{\overline x}) = D_8^d$. All parameters except $\alpha$ are the same as in Section \ref{subsec:bif-rel-equil}. In these tables, we follow the same agreement as in Table \ref{zero:eq:tab} except that we use a circle to indicate a Hopf bifurcation point and a rectangle to indicate a  steady-state bifurcation. In particular, an entry in a given cell indicates the number of unstable roots 
for the characteristic quasi-polynomial $\overline{\mathcal P}_j(\alpha,\lambda)$ associated with the isotypical component $U_{j,1}$ (shown in the left column) for the corresponding interval of $\alpha$-valus (shown in the upper row). The results presented in Subsection \ref{s6} follow from these tables.

\begin{table}[ht!]
{\setlength{\tabcolsep}{0.5em} 
\centering
\caption{Number of unstable eigenvalues for each  isotypical component along the branch of the relative equilibrium with $(D_8)$ symmetry (see item (i) on page \pageref{pageI})}
{\small
\begin{tabular}{|c|c|c|c|c|c|c|c|c|c|}
\cline{3-10}
\multicolumn{2}{c|}{} &\multicolumn{8}{c|}{Intervals for values of parameter $\alpha\; \cdot 10^2$ } \\ \cline{3-10}
\multicolumn{2}{c|}{} & \scriptsize $ [3.61,\; 3.69]$ &\scriptsize $[3.70,\; 3.85]$ & \scriptsize$[3.86,\; 4.01]$\scriptsize &\scriptsize $[4.02,\; 4.58]$ &\scriptsize $[4.59,\; 5]$ & \scriptsize$[5.01,\; 5.32]$ &\scriptsize $[5.33,\; 6.01]$ &\scriptsize $[6.02,\; 8.97]$ \\ \hline 
\multirow{6}{*}{\rot{Isotypical component}}&$U_{0,1}$ & 0 & 0 & 0 & 0 & 0 & 0 & 0 & \circled{2} \\ \cline{2-10}
&$U_{1, 1}$ & 0 & \squared{2} & \circled{6} & 6 & 6 & 6 & 6 & 6  \\ \cline{2-10}
&$U_{2, 1}$ & 0 & 0 & 0 & \squared{2} & 2 & 2 & \circled{6} & 6 \\ \cline{2-10}
&$U_{3, 1}$ & 0 & 0 & 0 & 0 & \squared{2} & 2 & 2 & 2 \\ \cline{2-10}
&$U_{4, 1}$ & 0 & 0 & 0 & 0 & 0 & \squared{1} & 1 & 1  \\ \cline{2-10}
&$ \bigoplus\limits_{j=0}^4 U_{j, 1} $& 0 & 2 & 6 & 8 & 10 & 11 & 15 & 17\\ \hline
\end{tabular}
}}
\end{table}


\begin{table}[ht!]
{\setlength{\tabcolsep}{0.5em}
\centering
\caption{Number of unstable eigenvalues for each  isotypical component along the branch of the relative equilibrium with $(\mathbb Z_8^{t_1})$ symmetry (see item (ii) on page \pageref{pageI})}
{\small
\begin{tabular}{|c|c|c|c|c|c|c|c|c|c|}
\cline{3-10}
\multicolumn{2}{c|}{} &\multicolumn{8}{c|}{Intervals for values of parameter $\alpha \; \cdot 10^2$ } \\ \cline{3-10}
 \multicolumn{2}{c|}{} &\scriptsize $[3.61,\; 3.65]$ &\scriptsize $[3.66,\; 3.98]$ &\scriptsize $[3.99,\; 4.15]$ &\scriptsize $[4.16,\; 4.2]$ &\scriptsize $[4.21,\; 6.39]$ &\scriptsize $6.40$ &\scriptsize $[6.41,\; 7.87]$ &\scriptsize $[7.88,\; 10.02]$ \\ \hline
\multirow{6}{*}{\rot{Isotypical component}} &$U_{0,1}$  & 0 & 0 & 0 & 0 & 0 & 0 & 0 & \circled{2} \\ \cline{2-10}
&$U_{1,1}$  & 2 & \circled{4} & 4 & \circled{6} & 6 & 6 & 6 & 6  \\ \cline{2-10}
&$U_{2,1}$ & 0 & \circled{2} & 2 & 2 & 2 & 2 & \circled{4} & 4  \\ \cline{2-10}
&$U_{3,1}$  & 0 & 0 & \circled{2} & \circled{4} & 4 & 4 & 4 & 4 \\ \cline{2-10}
&$U_{4,1}$  & 0 & 0 & 0 & 0 & \squared{1} & \circled{3} & 3 & 3 \\ \cline{2-10}
&$\bigoplus\limits_{j=0}^4 U_{j, 1} $ & 2 & 6 & 8 & 12 & 13 & 15 & 17 & 19 \\ \hline
\end{tabular}
}}
\end{table}

\begin{table}[ht!]
{\setlength{\tabcolsep}{0.5em}
\begin{center}
\caption{Number of unstable eigenvalues for each  isotypical component along the branch of the relative equilibrium with $(\mathbb Z_8^{t_2})$ symmetry (see item (iii) on page \pageref{pageI})}
{\small
\begin{tabular}{|c|c|c|c|c|c|c|}
\cline{3-7}
\multicolumn{2}{c|}{} &\multicolumn{5}{c|}{Intervals for values of parameter  $\alpha \; \cdot 10^2$ } \\ \cline{3-7}
 \multicolumn{2}{c|}{} &\scriptsize $[3.61,\; 5.48]$ &\scriptsize $[5.49,\;6.6]$ &\scriptsize \scriptsize$[6.61,\; 8.09]$ &\scriptsize $[8.10,\; 8.47]$ &\scriptsize $[8.48,\; 13.53]$ \\ \hline 
\multirow{6}{*}{\rot{Isotypical component}}&$U_{0,1}$ & 0 & 0 & 0 & 0 & \circled{2} \\ \cline{2-7}
&$U_{1,1}$ & 2 & 2 & \circled{4} & 4 & 4 \\ \cline{2-7}
&$U_{2,1}$ & 2 & \circled{4} & 4 & \circled{6} & 6 \\ \cline{2-7}
&$U_{3,1}$ & 2 & 2 & \circled{4} & 4 & 4 \\ \cline{2-7}
&$U_{4,1}$& 0 & 0 & 0 & 0 & \circled{2} \\ \cline{2-7}
&$\bigoplus\limits_{j=0}^4 U_{j, 1} $& 6 & 8 & 12 & 14 & 18 \\ \hline
\end{tabular}
}
\end{center}
}
\end{table}

\begin{table}[ht!]
{\setlength{\tabcolsep}{0.5em}
\centering
\caption{Number of unstable eigenvalues for each  isotypical component along the branch of the relative equilibrium with $(\mathbb Z_8^{t_3})$ symmetry (see item (iv) on page \pageref{pageI})}

{\small
\begin{tabular}{|c|c|c|c|c|c|c|c|c|c|c|}
\cline{3-11}
\multicolumn{2}{c|}{} &\multicolumn{9}{c|}{Intervals for values of parameter $\alpha \; \cdot 10^2$} \\ \cline{3-11}
\multicolumn{2}{c|}{} &\scriptsize$[3.61,\; 3.68]$ &\scriptsize $[3.69,\;3.84]$ &\scriptsize $3.85$ &\scriptsize $[3.86,\; 5.24]$ &\scriptsize $[5.25,\; 5.29]$ &\scriptsize $[5.3,\; 5.52]$ &\scriptsize $[5.53,\; 6.37]$ &\scriptsize $[6.38,\;8.65]$ &\scriptsize $[8.66,\; 9.23]$ \\ \hline 
\multirow{6}{*}{\rot{Isotypical component}}&$U_{0,1}$ & 0 & 0 & 0 & 0 & 0 & 0 & 0 & 0 & \circled{2}  \\ \cline{2-11}
&$U_{1,1}$ & 2 & 2 & 2 & \circled{4} & 4 & 4 & \circled{6} & 6 & 6 \\ \cline{2-11}
&$U_{2,1}$ & 2 & \circled{4} & 4 & 4 & \circled{6} & 6 & 6 & 6 & 6  \\ \cline{2-11}
&$U_{3,1}$ & 4 & 4 & \circled{6} & 6 & 6 & 6 & 6 & \circled{8} & 8 \\ \cline{2-11}
&$U_{4,1}$ & 2 & 2 & 2 & 2 & 2 & \circled{4} & 4 & 4 & 4 \\ \cline{2-11}
&$\bigoplus\limits_{j=0}^4 U_{j, 1} $ & 10 & 12 & 14 & 16 & 18 & 20 & 22 & 24 & 26 \\ \hline
\end{tabular}
}}
\end{table}

\begin{table}[ht!]
{\setlength{\tabcolsep}{0.5em}
\begin{center}
\caption{Number of unstable eigenvalues for each  isotypical component along the branch of the relative equilibrium with $(D_8^d)$ symmetry (see item (vi) on page \pageref{pageI})}
{\small
\begin{tabular}{|c|c|c|c|c|c|c|c|c|}
\cline{3-9}
\multicolumn{2}{c|}{} &\multicolumn{7}{c|}{Intervals for values of parameter $\alpha \; \cdot 10^2$ } \\ \cline{3-9}
\multicolumn{2}{c|}{} &\scriptsize $[3.62,\; 3.83]$ &\scriptsize $[3.84,\; 4.04]$ &\scriptsize $[4.05,\; 5.38]$ &\scriptsize $[5.39,\; 6.59]$ &\scriptsize $[6.6,\;7.3]$ &\scriptsize $[7.31,\; 7.56]$ &\scriptsize $[7.57,\; 13.55]$ \\ \hline 
\multirow{6}{*}{\rot{Isotypical component}}&$U_{0,1}$ & 0 & 0 & \circled{2} & 2 & 2 & 2 & 2 \\ \cline{2-9}
&$U_{1,1}$ & 4 & 4 & 4 & 4 & 4 & 4 & \circled{8} \\ \cline{2-9}
&$U_{2,1}$ & 4 & \circled{8} & 8 & 8 & 8 & 8 & 8 \\ \cline{2-9}
&$U_{3,1}$ & 4 & 4 & 4 & \circled{8} & 8 & 8 & 8 \\ \cline{2-9}
&$U_{4,1}$ & 4 & 4 & 4 & 4 & \circled{6} & \circled{4} & 4 \\ \cline{2-9}
&$\bigoplus\limits_{j=0}^4 U_{j, 1} $ & 16 & 20 & 22 & 26 & 28 & 26 & 30 \\ \hline
\end{tabular}
}
\end{center}}
\end{table}

\section{Appendix}\label{app:1}

If $\mathfrak W$ is a $G$-representation, then for any function $x:S^1 \to \mathfrak W$, the spatio-temporal symmetry of $x$ is a group $\mathfrak H < G \times S^1$ such that $g\cdot x(t-\theta)=x(t)$ for any $t\in \mathbb R/2\pi\mathbb Z \simeq S^1$ and any $(g,e^{i\theta}) \in \mathfrak H$. If $x$ is  non-constant, then $\mathfrak H$ has the structure of a graph of a homomorphism $\varphi : H \to S^1$, where $H$ stands some subgroup of $G$. 
To emphasize this nature of the group $\mathfrak H$, 
the following notation is commonly used:
$$
H^\varphi := \{(h,\varphi(h)\;:\;h\in H)\}.
$$
The group  $H^\varphi$ is called a {\it twisted} symmetry group with twisting homomorpism $\varphi$. 

Relative periodic solutions of our interest have symmetry groups which are subgroups of $\Gamma \times S^1 \times S^1$. 
Such a subgroup can be characterized by two twisting homomorphisms $\varphi : K \to S^1$ and $\psi : K^{\varphi} \to S^1$ 
for some subgroup  $K < \Gamma$.  However, in order to simplify our notations, instead of writing $K^{\varphi,\psi}$, we used
the bold symbol $\bold{K}^{\varphi}$ to distinguish it from the group $K^{\varphi}$ used for twisted symmetries of periodic solutions.
\subsection{Notations used for the twisted subgroups of $\mathcal H:=D_8\times  S^1$}\label{notation:0}
The following symbols are used for the twisted subgroups of  $\mathcal K$: we put $ \xi:=e^{\frac{\pi i}{4}}$ and $\kappa:=\left[\begin{array}{cc}1&0\\0&-1   \end{array}  \right]$
and denote
\begin{align*}
D_8&:=\{(\xi^k,1): k=0,1,\dots,7\}\cup \{(\xi^k\kappa,1) : k=0,1,\dots,7\},\\
{D _8^{d}}&:=\{(\xi^k,(-1)^{k}: k=0,1,\dots,7\}\cup\{(\xi^k\kappa,(-1)^k): k=0,1,\dots,7\},\\
\wt D^d_4&:=\{(1,1),(i,-1),(-1,1),(-i,-1),(\xi\kappa,1),(\xi i\kappa,-1),(-\xi\kappa,1),(-\xi i\kappa,-1)\},\\
D^d_4&:=\{(1,1),(i,-1),(-1,1),(-i,-1),(\kappa,1),(i\kappa,-1),(-\kappa,1),(-i\kappa,-1)\},\\
D^d_2&:=\{(1,1),(-1,-1),(\kappa,1),(-\kappa,-1)\},\\
\wt D_2^d&:=\{(1,1),(-1,-1),(\xi\kappa,1),(-\xi\kappa,-1)\},\\
{\mathbb Z_8^{t_1}}&:=\{(\xi^k,\xi^k: k=0,1,\dots,7\},\\ 
{\mathbb Z_8^{t_2}}&:=\{(\xi^k,\xi^{2k}): k=0,1,\dots,7\},\\
{\mathbb Z_8^{t_3}}&:=\{(\xi^k,\xi^{3k}): k=0,1,\dots,7\}.\\ 
\end{align*}

\subsection{Notations used for the twisted subgroups of $\mathcal K:=D_8\times \{1\}\times S^1$}\label{notation:1}
The following symbols are used for the twisted subgroups of  $\mathcal K$:
\begin{align*}\boldsymbol{D_8}&:=D_8\times \{1\}\times \{1\},\\
\boldsymbol{\mathbb Z_8^{t_1}}&:=\{(\xi^k,1,\xi^k)\in \mathcal K: k=0,1,\dots,7\},\quad \xi:=e^{\frac{\pi i}{4}},\\
\boldsymbol {D_2^d}&:=\{(1,1,1),(-1,1,-1),(\kappa, 1, 1), (-\kappa,1,-1)\} ,  \\
\boldsymbol{\widetilde{D}_2^d}&:=\{(1,1,1),(-1,1,-1),(\xi \kappa, 1, 1), (-\xi\kappa,1,-1)\}  , \\
\boldsymbol{\mathbb Z_8^{t_2}}&:=\{(\xi^k,1,\xi^{2k})\in \mathcal K: k=0,1,\dots,7\},\\
\boldsymbol{D_4^d}&:=\{(1,1,1),(i,1,-1)(-1,1,1),(-i,1,-1),(\kappa, 1, 1), (i\kappa,1,-1),\\
&\hskip.8cm (-\kappa,1,1),(-i\kappa, 1,-1))\} ,  \\
\boldsymbol{\widetilde{D}_4^d}&:=\{(1,1,1),(i,1,-1)(-1,1,1),(-i,1,-1),(\xi\kappa, 1, 1), (i\xi\kappa,1,-1),\\
&\hskip.8cm (-\xi\kappa,1,1),(-i\xi\kappa, 1,-1))\} ,  \\
\boldsymbol{\mathbb Z_8^{t_3}}&:=\{(\xi^k,1,\xi^{3k})\in \mathcal K: k=0,1,\dots,7\},\\ 
\boldsymbol{D_8^d}&:=\{(\xi^k,1,(-1)^k), \, \{(\xi^k\kappa ,1,(-1)^k)\in \mathcal H: k=0,1,\dots,7\}.
\end{align*}

\subsection{Notations used for the twisted subgroups of $\mathcal K:=\bz_8^{t_l}\times S^1$, $l=1,2,3$}\label{notation:2}
In this case, the following symbols are used for the twisted subgroups of  $\mathcal K$:
\begin{align*}
\boldsymbol{\mathbb Z_8}&:=\bz_8^{t_l}\times \{1\},\\
\boldsymbol{\mathbb Z_8^{t_1}}&:=\{(\xi^k,\xi^{lk},\xi^k)\in \mathcal K: k=0,1,\dots,7\},\quad \xi:=e^{\frac{\pi i}{4}},\\
\boldsymbol{\mathbb Z_8^{t_2}}&:=\{(\xi^k,\xi^{lk},\xi^{2k})\in \mathcal K: k=0,1,\dots,7\},\\
\boldsymbol{\mathbb Z_8^{t_3}}&:=\{(\xi^k,\xi^{lk},\xi^{3k})\in \mathcal K: k=0,1,\dots,7\},\\ 
\boldsymbol{\mathbb Z_8^{c}}&:=\{(\xi^k,\xi^{lk},(-1)^k)\in \mathcal K: k=0,1,\dots,7\}.
\end{align*}

\subsection{Notations used for the twisted subgroups of $\mathcal K:=D^d_8\times S^1$}\label{notation:3}
For this group, the following symbols are used for the twisted subgroups of  $\mathcal K$:
\begin{align*}\boldsymbol {D_8}&:=D^d_8\times \{1\},\\
\boldsymbol{\mathbb Z_8^{t_1}}&:=\{(\xi^k,(-1)^k,\xi^k)\in \mathcal K: k=0,1,\dots,7\},\quad \xi:=e^{\frac{\pi i}{4}},\\
\boldsymbol {D_2^d}&:=\{(1,1,1),(-1,1,-1),(\kappa, 1, 1), (-\kappa,1,-1)\} ,  \\
\boldsymbol{\widetilde{D}_2^d}&:=\{(1,1,1),(-1,1,-1),(\xi \kappa, -1, 1), (-\xi\kappa,-1,-1)\}  , \\
\boldsymbol{\mathbb Z_8^{t_2}}&:=\{(\xi^k,(-1)^k,\xi^{2k})\in \mathcal K: k=0,1,\dots,7\},\\
\boldsymbol{D_4^d}&:=\{(1,1,1),(i,1,-1)(-1,1,1),(-i,1,-1),(\kappa, 1, 1), (i\kappa,1,-1),\\
&\hskip.8cm (-\kappa,1,1),(-i\kappa, 1,-1))\} ,  \\
\boldsymbol{\widetilde{D}_4^d}&:=\{(1,1,1),(i,1,-1)(-1,1,1),(-i,1,-1),(\xi\kappa, -1, 1), (i\xi\kappa,-1,-1),\\
&\hskip.8cm (-\xi\kappa,-1,1),(-i\xi\kappa, -1,-1))\} ,  \\
\boldsymbol{\mathbb Z_8^{t_3}}&:=\{(\xi^k,(-1)^k,\xi^{3k})\in \mathcal K: k=0,1,\dots,7\},\\ 
\boldsymbol{D_8^d}&:=\{(\xi^k,(-1)^k,(-1)^k), \, \{(\xi^k\kappa ,(-1)^k,(-1)^k)\in \mathcal H: k=0,1,\dots,7\}.
\end{align*}


\section*{Acknowledgments}

The authors thank Andrei Vladimirov for the discussion of the laser model. The support of NSF through grant DMS-1413223 is greatfully acknowledged.

\end{document}